\documentclass{compositio}
\usepackage{amsmath}
\usepackage{amssymb}
\usepackage{epsfig}
\usepackage{graphicx}
\usepackage{color}
\usepackage{amscd}

\numberwithin{equation}{section}

\input xy
\xyoption{all}

\theoremstyle{plain}
\newtheorem{prop}{Proposition}[section]

\newtheorem{theo}[prop]{Theorem}
\newtheorem{coro}[prop]{Corollary}

\newtheorem{lemm}[prop]{Lemma}

\theoremstyle{definition}
\newtheorem{defi}[prop]{Definition}

\newtheorem{conj}[prop]{Conjecture}

\newtheorem{rema}[prop]{Remark}

\newtheorem{exam}[prop]{Example}

\newtheorem{obse}[prop]{Observation}

\def\ra{\rightarrow}

\def\bR{{\mathbb R}}

\def\Eff{\overline{\mathrm{Eff}}}
\def\Pic{\mathrm{Pic}}

\def\Mor{\mathrm{Mor}}
\def\Nef{\mathrm{Nef}}

\def\Pic{\mathrm{Pic}}

\makeatother
\makeatletter

\begin{document}

\title[Geometric Manin's Conjecture]{Geometric Manin's Conjecture and rational curves}

\author{Brian Lehmann}
\email{lehmannb@bc.edu}
\address{Department of Mathematics, Boston College, 1400 Commonwealth Ave, Chestnut Hill, MA, 02467}

\author{Sho Tanimoto}
\email{stanimoto@kumamoto-u.ac.jp}
\address{Department of Mathematics, Faculty of Science, Kumamoto University, Kurokami 2-39-1 Kumamoto 860-8555 Japan; 
Priority Organization for Innovation and Excellence, Kumamoto University}



\thanks{Lehmann is supported by NSF grant 1600875.  Tanimoto is partially supported by Lars Hesselholt's Niels Bohr professorship, and MEXT Japan, Leading Initiative for Excellent Young Researchers (LEADER).}

\begin{abstract}
Let $X$ be a smooth projective Fano variety over the complex numbers.  We study the moduli space of rational curves on $X$ using the perspective of Manin's Conjecture. In particular, we bound the dimension and number of components of spaces of rational curves on $X$.  We propose a Geometric Manin's Conjecture predicting the growth rate of a counting function associated to the irreducible components of these moduli spaces.
\end{abstract}

\maketitle

\section{Introduction} 
\label{secct:intro}

A Fano variety over $\mathbb{C}$ carries many rational curves due to the positivity of the anticanonical bundle (\cite{Mori84}, \cite{KMM92}, \cite{Campana92}).  The precise relationship between curvature and the existence of rational curves is quantified by Manin's Conjecture.  For an ample divisor $L$ on a Fano variety $X$, the constants $a(X,L)$ and $b(X,L)$ of \cite{BM} compare the positivity of $K_{X}$ and $L$. 
Manin's Conjecture predicts that the asymptotic behavior of rational curves on $X$ as the $L$-degree increases is controlled by these geometric constants.  This point of view injects techniques from the minimal model program to the study of spaces of rational curves.

Batyrev gave a heuristic for Manin's Conjecture over finite fields that depends on three assumptions (see \cite{Tsc09} or \cite{Bou} for Batyrev's heuristic):
\begin{enumerate}
\item after removing curves that lie on a closed subset, moduli spaces of rational curves have the expected dimension;
\item the number of components of moduli spaces of rational curves whose class is a nef integral $1$-cycle is bounded above;
\item the \'etale cohomology of moduli spaces of rational curves enjoys certain homological stability properties. (The idea to use homological stability in Batyrev's heuristic is due to Ellenberg and Venkatesh, see, e.g., \cite{EV05}.)
\end{enumerate} 

In this paper, we investigate the plausibility of the first two assumptions for complex varieties.  We prove that the first assumption holds for any smooth Fano variety.  The second assumption fails in general: the number of components can grow polynomially as the degree of the $1$-cycle grows.  Thus we proceed in two different directions.  First, it is conjectured by Batyrev that there is a polynomial upper bound on the growth in number of components, and we make partial progress toward this conjecture.  Second, we explain how to modify the conjecture in order to discount the ``extra'' components and recover Batyrev's heuristic.  Our proposal can be seen as a geometric analogue of Peyre's thin set version of Manin's Conjecture.

\subsection{Moduli of rational curves}

Let us discuss the contents of our paper in more detail.
Let $X$ be a smooth projective uniruled variety and let $\Eff^{1}(X)$ denote the pseudo-effective cone of divisors.  Suppose $L$ is a nef $\mathbb{Q}$-Cartier divisor on $X$.  When $L$ is big, define the Fujita invariant (which we will also call the $a$-invariant) by
\begin{equation*}
a(X,L) := \min \{ t\in \bR \mid t[L] + [K_X] \in \Eff^{1}(X) \}.
\end{equation*}
When $L$ is not big, we formally set $a(X,L) = \infty$.  When $X$ is singular, we define the $a$-invariant by pulling $L$ back to a resolution of $X$.

\begin{theo} \label{intro:expecteddim}
Let $X$ be a smooth projective weak Fano variety and set $L = -K_{X}$.  Let $V \subsetneq X$ be the proper closed subset which is the Zariski closure of all subvarieties $Y$ such that $a(Y,L|_{Y}) > a(X,L)$.  Then any component of $\mathrm{Mor}(\mathbb{P}^{1},X)$ parametrizing a curve not contained in $V$ will have the expected dimension and will parametrize a dominant family of curves.
\end{theo}

Assuming standard conjectures about rational curves, the converse implication is also true: a subvariety with higher $a$-value will contain families of rational curves with dimension higher than the expected dimension in $X$.  In this way the $a$-invariant should completely control the expected dimension of components of $\mathrm{Mor}(\mathbb{P}^{1},X)$.  Furthermore, an analogous statement holds for any uniruled $X$ and any big and nef $L$ provided we restrict our attention to curves with vanishing intersection against $K_{X} + a(X,L)L$. 

Theorem \ref{intro:expecteddim} is significant for two reasons.  The first is that $V$ is a proper subset of $X$; this is the main theorem of \cite{HJ16}.  The second is that Theorem \ref{intro:expecteddim} gives an explicit description of the closed set $V$.  In practice, one can use techniques from adjunction theory or the minimal model program to calculate $V$.

\begin{exam}
In Example \ref{exam:quartichypersurface} we show that if $X$ is any smooth quartic hypersurface of dimension $\geq 5$ then the exceptional set $V$ in Theorem~\ref{intro:expecteddim} is empty so that every component of $\mathrm{Mor}(\mathbb{P}^{1},X)$ has the expected dimension.  The same approach gives a quick proof of a result of \cite{CS09} showing an analogous property for cubic hypersurfaces.  Note that for a quartic hypersurface the components of the Kontsevich moduli space of stable maps need not have the expected dimension (see \cite{CS09}), so the method in \cite{CS09} does not apply to this case.
\end{exam}

\begin{exam}
Let $X$ be a smooth Fano threefold with index $2$ and Picard rank $1$.  By \cite[Proposition 6.5, 6.8, and 6.11]{LTT14}, the exceptional set $V$ in Theorem~\ref{intro:expecteddim} is empty so that every component of $\mathrm{Mor}(\mathbb{P}^{1},X)$ has the expected dimension and parametrizes a dominant family of curves.  
\end{exam}

The main outstanding question concerning $\mathrm{Mor}(\mathbb{P}^{1},X)$ is the number of components. Batyrev first conjectured that the number of components grows polynomially with the degree of the curve.  In fact we expect that the growth is controlled in a precise way by another invariant in Manin's Conjecture: the $b$-invariant (see Definition~\ref{defi: invariant b}).  We prove a polynomial growth bound for components satisfying an additional hypothesis: 

\begin{theo} \label{theo:polynomialbounds}
Let $X$ be a smooth projective uniruled variety and let $L$ be a big and nef $\mathbb{Q}$-Cartier divisor on $X$.  Fix a positive integer $q$ and let $\overline{M} \subset \overline{M}_{0,0}(X)$ denote the union of all components which contain a chain of free curves whose components have $L$-degree at most $q$.  There is a polynomial $P(d)$ which is an upper bound for the number of components of $\overline{M}$ of $L$-degree at most $d$.
\end{theo}

Theorem \ref{theo:polynomialbounds} should be contrasted with bounds on the number of components of the Chow variety which are exponential in $d$ (\cite{Manin95}, \cite{Kollar}, \cite{guerra99}, \cite{hwang05}).

It is natural to wonder whether free curves of sufficiently high degree can always be deformed (as a stable map) to a chain of free curves of smaller degree.  Although this property seems subtle to verify, we are not aware of any Fano variety for which it fails.  We are able to verify it in some new situations for Fano varieties of small dimension.  By applying general theory, we can then understand the behavior of rational curves by combining an analysis of the $a$ and $b$ invariants with a few computations in low degree.

\begin{theo}
Let $X$ be a smooth Fano threefold of index $2$ with $\mathrm{Pic}(X)= \mathbb Z H$.  
We show that if $H^3 \geq 3$, or $H^3 = 2$ and $X$ is general in its moduli, then $\mathrm{Mor}(\mathbb{P}^{1},X)$ has two components of any anticanonical degree $2d \geq 4$: the family of $d$-fold covers of lines, and a family of irreducible degree $2d$ curves.
\end{theo}

In fact our proof shows that $\overline{M}_{0,0}(X)$ has the same components; this implies, for example, that certain Gromov-Witten invariants on $X$ are enumerative.
Previously such results were known for cubic threefolds by work of Starr (see \cite[Theorem 1.2]{CS09}) and for complete intersections of two quadrics by \cite{Cas04}, and our method significantly simplifies the proofs of these papers using the analysis of $a, b$ invariants.

\subsection{Manin-type bound}

Using the previous results, we prove an upper bound of Manin-type for the moduli space of rational curves.   Suppose that $X$ is a smooth projective uniruled variety and that $L$ is a big and nef divisor.  As a first attempt at a counting function, fix a variable $q$ and define
$$N(X, L, q, d) = \sum_{i=1}^{d} \sum_{W \in \mathcal{S}_{i}} q^{\dim W}$$
where $\mathcal{S}_{i}$ denotes the set of components $M \subset \mathrm{Mor}(\mathbb P^1, X)$ satisfying:
\begin{enumerate}
\item $M$ generically parametrizes free curves.
\item The curves parametrized by $M$ have $L$-degree $i \cdot r(X,L)$, where $r(X,L)$ is the minimal positive number of the form $L \cdot \alpha$ for a $\mathbb{Z}$-curve class $\alpha$.
\item The curves parametrized by $M$ satisfy $(K_{X} + a(X,L)L) \cdot C = 0$.
\end{enumerate}
This is not quite the correct definition; as usual in Manin's Conjecture one must remove the contributions of an ``exceptional set.''  In the number theoretic setting one must remove a thin set of points to obtain the expected growth rate.  An analogous statement is true in our geometric setting as well, and in Section \ref{section:statementofmc} we give a precise formulation of which components should be included in the definition of $N(X,L,q,d)$.

After modifying the counting function in this way, we can prove an asymptotic upper bound.  For simplicity we only state a special case:

\begin{theo} \label{theo:firstmc}
Let $X$ be a smooth projective Fano variety.  Fix $\epsilon > 0$; then for sufficiently large $q$
\begin{equation*}
N(X,-K_X, q,d) = O\left(q^{dr(X,-K_{X})(1+\epsilon)} \right).
\end{equation*}
\end{theo}

In the literature there are several examples of Fano varieties for which the components of $\Mor(\mathbb{P}^{1},X)$ have been classified.  In every  example we know of the counting function has the asymptotic behavior predicted by Manin's Conjecture.

\begin{exam} 
Let $X$ be a smooth del Pezzo surface of degree $\geq 2$ which admits a $(-1)$-curve.  Using \cite{Testa09}, Example \ref{exam:delpezzosurface} shows that
\begin{equation*}
N(X,-K_{X},q,d) \sim \frac{q^{2} \alpha(X,L)}{1-q^{-1}} q^{d} d^{\rho(X)-1}.
\end{equation*}
where $\alpha(X,L)$ is the volume of a polytope defined in Definition \ref{defi:alphaconstant} and $\rho(X)$ is the Picard rank of $X$.
\end{exam}


\acknowledgements{
We thank Morten Risager for his help regarding height zeta functions and Chen Jiang for suggesting we use the results of H\"oring.
We also would like to thank Tony V\'arilly-Alvarado for his help to improve the exposition.
We thank the anonymous referees for detailed suggestions to improve the exposition of the paper.}

\section{Preliminaries}

Throughout we work over an algebraically closed field of characteristic $0$.  Varieties are irreducible and reduced.

For $X$ a smooth projective variety we let $N^{1}(X)$ denote the space of $\mathbb{R}$-divisors up to numerical equivalence.  It contains a lattice $N^{1}(X)_{\mathbb{Z}}$ consisting of classes of Cartier divisors.  We let $\Eff^{1}(X)$ and $\Nef^{1}(X)$ denote the pseudo-effective and nef cones of divisors respectively; their intersections with $N^{1}(X)_{\mathbb{Z}}$ are denoted $\Eff^{1}(X)_{\mathbb{Z}}$ and $\Nef^{1}(X)_{\mathbb{Z}}$.  Dually, $N_{1}(X)$ denotes the space of curves up to numerical equivalence with natural lattice $N_{1}(X)_{\mathbb{Z}}$.  $\Eff_{1}(X)$ and $\Nef_{1}(X)$ denote the pseudo-effective and nef cones of curves, containing lattice points $\Eff_{1}(X)_{\mathbb{Z}}$ and $\Nef_{1}(X)_{\mathbb{Z}}$.

Suppose that $f,g: \mathbb{N} \to \mathbb{R}$ are two positive real valued functions.  We use the symbol $f(d) \sim g(d)$ to denote ``asymptotically equal'':
\begin{equation*}
\lim_{d \to \infty} \frac{f(d)}{g(d)} = 1
\end{equation*}
We will also use the standard ``big-O'' notation when we do not care about constant factors.

Certain kinds of morphisms play a special role in Manin's Conjecture:

\begin{defi}
We say that a morphism of projective varieties $f: Y \to X$ is a thin morphism if $f$ is generically finite onto its image but is not both dominant and birational.
\end{defi}

\section{Geometric invariants $a, b$} \label{absection}

\subsection{Background}
We recall the definitions of the $a$ and $b$ invariants studied in \cite{HTT15}, \cite{LTT14}, \cite{HJ16}, \cite{LT16}.
These invariants also play a central role in the study of cylinders, see, e.g., \cite{CPW16}.

\begin{defi}\cite[Definition 2.2]{HTT15}
\label{defi: Fujita invariant}
Let $X$ be a smooth projective variety 
and let $L$ be a big and nef $\mathbb{Q}$-divisor on $X$. 
The {\it Fujita invariant} is
$$
a(X, L) := \min \{ t\in \bR \mid t[L] + [K_X] \in \Eff^{1}(X) \}.
$$
If $L$ is not big, we set $a(X,L) = \infty$.
\end{defi}

By \cite[Proposition 7]{HTT15}, $a(X, L)$ does not change when pulling back $L$ by a birational map. 
Hence, we define the Fujita invariant for a singular projective variety $X$ 
by pulling back to a smooth resolution $\beta : \tilde{X} \ra X$:
$$
a(X, L):= a(\tilde{X}, \beta^*L).
$$
This definition does not depend on the choice of $\beta$.  
It follows from \cite{BDPP} that $a(X,L)$ is positive if and only if $X$ is uniruled. 

\begin{defi}\cite[Definition 2.8]{HTT15}
\label{defi: invariant b}
Let $X$ be a smooth projective variety such that $K_X$ is not pseudo effective. Let $L$ be a big and nef $\mathbb{Q}$-divisor on $X$. We define $b(X, L)$ to  be the codimension of the minimal supported face of $\Eff^{1}(X)$ containing the class $a(X, L)[L] + [K_X]$. 
\end{defi}

Again, this is a birational invariant (\cite[Proposition 9]{HTT15}), 
and we define $b(X, L)$ for a singular variety $X$ by taking a smooth resolution $\beta : \tilde{X} \ra X$ and setting
$$
b(X, L) := b(\tilde{X}, \beta^*L).
$$
This definition does not depend on the choice of $\beta$.  It turns out $b$ has a natural geometric interpretation in terms of Picard ranks (see \cite[Corollary 3.9 and Lemma 3.10]{LTT14}).



\subsection{Compatibility statements} \label{section:abcompatibility} 
Let $X$ be a smooth projective variety and let $L$ be a big and nef $\mathbb{Q}$-divisor on $X$.  Suppose that $f: Y \to X$ is a thin morphism.  It will be crucial for us to understand when
\begin{equation*}
(a(Y,f^{*}L),b(Y,f^{*}L)) > (a(X,L),b(X,L))
\end{equation*}
in the lexicographic order.  We say that $f$ {\it breaks the weakly balanced condition} when such an inequality holds.  When $f$ only induces an inequality $\geq$, we say that $f$ {\it breaks the balanced condition}.  

The case when $f: Y \to X$ is the inclusion of a subvariety is of particular importance.  The following theorem of \cite{HJ16} describes when the $a$-invariant causes an inclusion $f: Y \to X$ to break the balanced condition.  The proof relies upon the recent boundedness statements of Birkar.

\begin{theo}[\cite{LTT14} Theorem 4.8 and \cite{HJ16} Theorem 1.1] \label{theo:boundedavalues}
Let $X$ be a smooth uniruled projective variety and let $L$ be a big and nef $\mathbb{Q}$-divisor on $X$.  Let $V$ denote the union of all subvarieties $Y$ such that $a(Y,L|_{Y}) > a(X,L)$.  Then $V$ is a proper closed subset of $X$ and its components are precisely the maximal elements in the set of subvarieties with higher $a$-value.
\end{theo}

\begin{proof}
Since \cite{birkar16b} has settled the Borisov-Alexeev-Borisov Conjecture, \cite[Theorem 4.8]{LTT14} proves that the closure of $V$ is a proper closed subset of $X$.  In fact the proof gives a little bit more: every component of $V$ is dominated by a family of subvarieties with $a$-value higher than $X$.  By \cite[Proposition 4.1]{LTT14} every component of $V$ will also have higher $a$-value than $X$.
\end{proof}

The other important case to consider is when $f: Y \to X$ is a dominant map.  For convenience we formalize this situation into a definition.

\begin{defi}
Let $X$ be a smooth uniruled projective variety and let $L$ be a big and nef $\mathbb{Q}$-divisor on $X$.  
We say that a morphism from a smooth projective variety $f:Y \rightarrow X$ is {\it an $a$-cover} if (i) $f$ is a dominant thin morphism and (ii) $a(Y, f^*L) =  a(X, L)$.
\end{defi}

\subsection{Face contraction} \label{section:facecontracting}

The following definitions encode a slightly more refined version of the $b$ invariant.

\begin{defi}
Let $X$ be a smooth uniruled projective variety and let $L$ be a big and nef $\mathbb{Q}$-divisor on $X$.   We let $F(X,L)$ denote the face of $\Nef_{1}(X)$ consisting of those curve classes $\alpha$ satisfying $(K_{X} + a(X,L)L) \cdot \alpha = 0$.
\end{defi}

When $f: Y \to X$ is an $a$-cover, the Riemann-Hurwitz formula implies that the pushforward $f_{*}: N_{1}(Y) \to N_{1}(X)$ maps $F(Y,f^{*}L)$ to $F(X,L)$.

\begin{defi}
Let $X$ be a smooth uniruled projective variety and let $L$ be a big and nef $\mathbb{Q}$-divisor on $X$.  
We say that a morphism $f: Y \to X$ is \emph{face contracting} if $f$ is an $a$-cover and the map $f_{*}: F(Y,f^{*}L) \to F(X,L)$ is not injective.
\end{defi}

Recall that the dimensions of the faces $F(Y,f^{*}L), F(X,L)$ are respectively $b(Y,f^{*}L), b(X,L)$.  Thus, if $f$ breaks the weakly balanced condition then it is also automatically face contracting.  However, $F(Y,f^{*}L)$ need not surject onto $F(X,L)$ and so not all face contracting morphisms break the weakly balanced condition.

\begin{exam} \label{exam:facecontracting}
\cite{MZ88} identifies a del Pezzo surface $X'$ with canonical singularities which admits a finite cover $f': Y' \to X'$ which is \'etale in codimension $1$ and such that $\rho(X') = 1$ and $\rho(Y')=2$.  Let $f: Y \to X$ be a resolution of this map and set $L = -K_{X}$.  \cite[Theorem 6.1]{LT16} shows that $f$ does not break the weakly balanced condition.  Nevertheless, we claim that the pushforward $f_{*}$ contracts $F(Y,f^{*}L)$ to a face $F$ of smaller dimension.  

It suffices to find two different classes in $F(Y,f^{*}L)$ whose images under $f_{*}$ are the same.  Since $f'_{*}: N_{1}(Y') \to N_{1}(X')$ drops the Picard rank by $1$, there are ample curve classes $\beta$ and $\beta'$ on $Y'$ whose images under $f'_{*}$ are the same.  Let $\alpha$ and $\alpha'$ be their pullbacks in $N_{1}(Y)$; we show that $f_{*}\alpha = f_{*}\alpha'$. 

Note that every curve exceptional for the birational morphism $X \to X'$ pulls back under $f$ to a union of curves exceptional for the morphism $Y \to Y'$.  Applying the projection formula to $f$, we deduce that $f_{*}\alpha$ and $f_{*}\alpha'$ have vanishing intersection against every exceptional curve for $X \to X'$.  Since $f_{*}\alpha$ and $f_{*}\alpha'$ also push forward to the same class on $X'$ by construction, we conclude that $f_{*}\alpha = f_{*}\alpha'$.
\end{exam}

Face contracting morphisms are important for understanding the leading constant in Manin's Conjecture.  Fix a number field $K$, and suppose that $f: Y \to X$ is a dominant generically finite morphism of smooth projective varieties over $K$ with equal $a,b$-values.  Manin's conjecture predicts that the growth rate of rational points of bounded height is the same on $X$ and $Y$.  Thus to obtain the correct Peyre's constant for the rate of growth of rational points one must decide whether or not to include $f(Y(K))$ in the counting function.

Face contraction gives us a geometric criterion to distinguish whether we should include the point contributions from $Y$.  When $X$ is a Fano variety with an anticanonical polarization, the key situation to understand is when $f : Y\rightarrow X$ is Galois and $a(X, L)f^*L + K_Y$ has Iitaka dimension $0$.  After replacing $Y$ by a birational modification, we may assume that any birational transformation of $Y$ over $X$ is regular.  In this situation \cite[Proposition 8.4]{LT16} gives a geometric condition determining whether $f$ and its twists give the entire set of rational points (and thus whether or not these contributions must be removed).  It turns out that the geometric condition of \cite[Proposition 8.4]{LT16} is equivalent to being face contracting.


\subsection{Varieties with large $a$-invariant}

The papers \cite{Fujita89}, \cite{Horing10}, \cite{Andreatta13} give a classification of varieties with large $a$-invariant in the spirit of the Kobayashi-Ochiai classification.  The following two results are immediate consequences of \cite[Proposition 1.3]{Horing10} which classifies the smooth projective varieties and big and nef Cartier divisors satisfying $a(X,L) > \dim(X) - 1$.

\begin{lemm} \label{lemm:adjunctionresult}
Let $Y$ be a smooth projective variety of dimension $r$ and let $H$ be a big and nef divisor on $Y$.  Suppose that $a(Y,H) > r$.  Then $H^{r} = 1$.
\end{lemm}



\begin{lemm} \label{lemm:abiggerthanr-1}
Let $Y$ be a smooth projective variety of dimension $r\geq 2$ and let $H$ be a big and basepoint free divisor on $Y$.  Suppose that $a(Y,H) > r-1$ and that $\kappa(K_{Y} + a(Y,H)H) = 0$.  Then $H^{r} \leq 4$.  Furthermore, if $H^{r} = 4$ then a surface $S$ defined by a general complete intersection of elements of $H$ admits a birational morphism to $\mathbb{P}^{2}$ and $H|_{S}$ is the pullback of $\mathcal{O}(2)$.
\end{lemm}




\section{Expected dimension of rational curves}

We let $\mathrm{Mor}(\mathbb{P}^{1},X)$ denote the quasi-projective scheme parametrizing maps from $\mathbb{P}^{1}$ to $X$ as constructed by \cite{Grothendieck95}.

\begin{defi}
Let $X$ be a smooth projective variety and let $\alpha \in \Eff_{1}(X)_{\mathbb{Z}}$.  We let $\mathrm{Mor}(\mathbb{P}^{1},X,\alpha)$ denote the set of components of $\mathrm{Mor}(\mathbb{P}^{1},X)$ parametrizing curves of class $\alpha$.

Given an open subset $U \subset X$, $\mathrm{Mor}_{U}(\mathbb{P}^{1},X,\alpha)$ denotes the sublocus of $\Mor(\mathbb{P}^{1},X,\alpha)$ which parametrizes curves meeting $U$.
\end{defi}

Let $W$ be an irreducible component of $\mathrm{Mor}(\mathbb P^1, X,\alpha)$.  The ``expected dimension'' of $W$ is $-K_{X} \cdot \alpha + \dim X$.  It turns out that we always have an inequality 
\begin{equation*}
\dim W \geq -K_{X} \cdot \alpha + \dim X
\end{equation*}
and when $W$ parametrizes a dominant family of curves then equality is guaranteed (\cite{Kollar}).

\begin{prop} \label{prop:expdim}
Let $X$ be a smooth projective uniruled variety and let $L$ be a big and nef $\mathbb{Q}$-divisor on $X$.
Let $W$ be an irreducible component of $\mathrm{Mor}(\mathbb P^1, X,\alpha)$ satisfying $(K_{X} + a(X,L)L) \cdot \alpha = 0$, and let $\pi: \mathcal C \rightarrow W$ be the corresponding family of irreducible rational curves with the evaluation map $s: \mathcal C \rightarrow X$.  Set $Z = \overline{s(\mathcal C)}$.
\begin{enumerate}
\item Suppose that the dimension of $W$ is greater than the expected dimension, i.e., 
\[
\dim W > -K_X \cdot \alpha + \dim X. 
\]
Then $a(Z, L|_Z) > a(X, L)$.
\item Suppose that $Z \neq X$.  Then $a(Z, L|_Z) > a(X, L)$.
\end{enumerate}
\end{prop}


\begin{proof}
If $a(Z,L|_{Z}) = \infty$ then both statements are true so we may suppose otherwise.  Let $f:Y \rightarrow Z$ be a resolution of singularities.
By taking strict transforms of curves we obtain a family of curves on $Y$, $\mathcal C^\circ \rightarrow W^\circ $, where $W^{\circ}$  is an open subset of the reduced space underlying $W$, and with an evaluation map $s : \mathcal C^\circ \rightarrow Y$.  Let $C$ denote an irreducible curve parametrized by $W^{\circ}$.
Since $W^{\circ}$ is contained in an irreducible component of $\mathrm{Mor}(\mathbb P^1, Y)$ parametrizing curves which dominate $Y$, we have
\[
\dim(W^\circ) \leq -K_Y \cdot C + \dim Y.
\]
The dimension of $W$ is always at least the expected dimension, so $-K_X \cdot f_{*}C  + \dim X \leq -K_Y \cdot C + \dim Y$.  By assumption either this inequality is strict or $\dim Y < \dim X$, and in either case
\[
(K_Y - f^{*}K_X) \cdot C < 0.
\]
Since $(K_{X} + a(X,L)L)|_{Z} \cdot f_{*}C = 0$, we can equally well write $(K_{Y} + a(X,L)L|_{Y}) \cdot C < 0$.  Since $C$ deforms to cover $Y$, $K_Y+a(X,L)f^{*}L$ is not pseudo effective.
This implies that $a(Y, L) > a(X,L)$.
\end{proof}

\begin{exam}
There is of course no analogous statement away from the face of curve classes vanishing against $K_{X} + a(X,L)L$.  Consider for example a K3 surface $S$ containing infinitely many $-2$ curves and let $X = \mathbb{P}^{1} \times S$.  For any big and nef $\mathbb{Q}$-divisor $L$, the divisor $K_{X} + a(X,L)L$ will be the pullback of a divisor on $S$.
Let $C$ be a $(-2)$-curve in some fiber over $\mathbb{P}^{1}$.  Then the component of $\Mor(\mathbb{P}^{1},X)$ corresponding to $C$ has dimension $4 > -K_{X} \cdot C + 3$.  Note however that $C$ has positive intersection against $K_{X} + a(X,L)L$ for any big and nef divisor $L$.
\end{exam}


\begin{theo} \label{theo:closednotfree}
Let $X$ be a smooth projective uniruled variety and let $L$ be a big and nef $\mathbb{Q}$-divisor.  Let $U$ be the Zariski open subset which is the complement of the closure of all subvarieties $Y \subset X$ satisfying $a(Y,L|_{Y}) > a(X,L)$.  Suppose that $\alpha \in \Nef_{1}(X)_{\mathbb{Z}}$ satisfies $(K_{X} + a(X,L)L) \cdot \alpha = 0$.  If $\mathrm{Mor}_U(\mathbb P^1, X, \alpha)$ is non-empty, then
\[
\dim \mathrm{Mor}_U(\mathbb P^1, X, \alpha) = -K_X \cdot \alpha + \dim X
\]
and every component of $\dim \mathrm{Mor}_U(\mathbb P^1, X, \alpha)$ parametrizes a dominant family of rational curves.
\end{theo}

\begin{proof}
By Theorem~\ref{theo:boundedavalues} there is a closed proper subset $V \subset X$ such that for every $Z \not \subset V$, $a(Z,L) \leq a(X,L)$.  Then apply Proposition \ref{prop:expdim}.
\end{proof}

\begin{rema}
\label{rema:non-nefclass}
Our proof actually shows that for any $\alpha \in \overline{\mathrm{Eff}}_1(X)_{\mathbb Z} \setminus \Nef_{1}(X)_{\mathbb{Z}}$, the space $\mathrm{Mor}_U(\mathbb P^1, X, \alpha)$ is empty. Indeed, suppose that it is not empty. Then there is an irreducible curve $C$ parametrized by $\mathrm{Mor}_U(\mathbb P^1, X, \alpha)$. Let $Z$ be the subvariety covered by deformations of $C$. Since $U\cap Z \neq \emptyset$, we have $a(Z, L|_{Z})\leq a(X, L)$. By Proposition \ref{prop:expdim}, $C$ must deform to cover $X$, i.e., $X=Z$. This means that $\alpha \in \Nef_{1}(X)_{\mathbb{Z}}$, a contradiction.
\end{rema}

The most compelling special case is:

\begin{theo} \label{theo:fanocase}
Let $X$ be a smooth projective weak Fano variety.  Let $U$ be the Zariski open subset which is the complement of the closure of all subvarieties $Y \subset X$ satisfying $a(Y,-K_{X}|_{Y}) > a(X,-K_{X})$. 
Then for any $\alpha \in \Nef_{1}(X)_{\mathbb{Z}}$ we have
\[
\dim \mathrm{Mor}_U(\mathbb P^1, X, \alpha) = -K_X \cdot \alpha + \dim X
\]
if it is not empty, and every component of $\dim \mathrm{Mor}_U(\mathbb P^1, X, \alpha)$ parametrizes a dominant family of rational curves. 
\end{theo}

\begin{exam}
If $X$ is not Fano, it is of course possible that non-dominant families of rational curves sweep out a countable collection of proper subvarieties, as in the blow-up of $\mathbb{P}^{2}$ at nine very general points.  
\end{exam}

Theorem \ref{theo:closednotfree} gives a new set of tools for understanding families of rational curves via adjunction theory.  Hypersurfaces are perhaps the most well-known source of examples of families of rational curves: we have an essentially complete description of the components of the moduli space of rational curves for general Fano hypersurfaces (\cite{HRS04}, \cite{BK13}, \cite{RY16}).  We briefly illustrate Theorem \ref{theo:closednotfree} by discussing results which hold for \emph{all} smooth hypersurfaces of a given degree.

\begin{exam}[\cite{CS09}] \label{exam:cubichypersurface}
Let $X$ be a smooth cubic hypersurface of dimension $n\geq 3$.  Let $H$ denote the hyperplane class on $X$, so that $K_{X} = -(n-1)H$ and $a(X,H) = n-1$.  We show that $X$ does not contain any subvariety with higher $a$-value, so that every family of rational curves has the expected dimension.  This recovers a result of \cite{CS09}.

Let $Y$ be the resolution of a subvariety of $X$.  \cite[Proposition 2.10]{LTT14} shows that the largest possible $a$-invariant for a big and nef divisor on a projective variety $Y$ is $\dim(Y) + 1$.  Thus if $Y$ has codimension $\geq 2$ then $a(Y,H) \leq n-1$.  If $Y$ has codimension $1$, Lemma \ref{lemm:adjunctionresult} shows that $a(Y,H) \leq n-1$ unless the $H$-degree of $Y$ is $1$.  But a smooth cubic hypersurface of dimension $\geq 3$ can not contain any codimension $1$ linear spaces, showing the claim.
\end{exam}

To our knowledge the following example has not been worked out explicitly in the literature.

\begin{exam} \label{exam:quartichypersurface}
Let $X$ be a smooth quartic hypersurface of dimension $n \geq 5$.  Let $H$ denote the hyperplane class on $X$, so that $a(X,H) = n-2$.  We prove that $X$ does not contain any subvariety with higher $a$-value, so that every family of rational curves has the expected dimension.  Suppose that $Y \subset X$ is a subvariety of codimension $\geq 3$.  Just as in Example \ref{exam:cubichypersurface}, we can immediately deduce that $a(Y,H) \leq a(X,H)$.

Next suppose that $Y \subset X$ has codimension $2$.  Applying Lemma \ref{lemm:adjunctionresult}, we see that $a(Y,H) \leq a(X,H)$ unless possibly if $H$ is a linear codimension $2$ space.  But this is impossible in our dimension range.

Finally, suppose there were a divisor $Y \subset X$ satisfying $a(Y,H) > a(X,H)$.  
If $\kappa(K_{Y} + a(Y,H)H) > 0$, then by \cite[Theorem 4.5]{LTT14} $Y$ is covered by subvarieties of smaller dimension with the same $a$-value, an impossibility by the argument above.  If $\kappa(K_{Y} + a(Y,H)H) = 0$, we may apply Lemma \ref{lemm:abiggerthanr-1} to see that $H|_{Y}^{r-1} \leq 4$.  By the Lefschetz hyperplane theorem, the only possibility is that $Y$ is the intersection of $X$ with a hyperplane and $H|_{Y}^{r-1}=4$.  Let $\widetilde{Y}$ denote a resolution of $Y$ and let $S$ be a surface which is a general complete intersection of members of $H$ on $\widetilde{Y}$.  Again applying Lemma \ref{lemm:abiggerthanr-1}, we see that the morphism defined by a sufficiently high multiple of $H|_{S}$ should define a map to $\mathbb{P}^{2}$.  In our situation it defines a map to a (possibly singular) reduced irreducible quartic surface, a contradiction.  (Note that this singular quartic must be normal because of \cite[Lemma 6.14]{LTT14}.) Thus $a(Y,H) \leq a(X,H)$ in every case.
\end{exam}

\section{Number of components} \label{section:numberofcomponents}
In this section we study the following conjecture of Batyrev:

\begin{conj}[Batyrev's conjecture] \label{conj:numberofcomponents}
Let $X$ be a smooth projective uniruled variety and let $L$ be a big and nef $\mathbb{Q}$-divisor on $X$.  For a numerical class $\alpha \in \Nef_{1}(X)_{\mathbb{Z}}$, let $h(\alpha)$ denote the number of components of $\mathrm{Mor}(\mathbb P^1, X, \alpha)$ that generically parametrize free curves. There is a polynomial $P(d) \in \mathbb{Z}[d]$ such that $h(\alpha) \leq P(L \cdot \alpha)$ for all $\alpha$.
\end{conj}

We prove a polynomial upper bound for components satisfying certain extra assumptions.  We will also give a conjectural framework for understanding the number of components that is motivated by Manin's Conjecture.


\subsection{Conjectural framework}

We expect that polynomial growth as in Conjecture \ref{conj:numberofcomponents} should arise from dominant maps $f: Y \to X$ which are face contracting.  In this case there will be many nef curve classes on $Y$ which are identified under pushforward to $X$, yielding many different components of the space of morphisms.  

\begin{exam}
We use an example considered by \cite{LeRudulier} in the number theoretic setting.  Set $S = \mathbb{P}^{1} \times \mathbb{P}^{1}$ and let $X = \mathrm{Hilb}^{2}(S)$, a weak Fano variety of dimension $4$.  Let $Y$ denote the blow up of $S \times S$ along the diagonal and let $f: Y \to X$ denote the natural $2:1$ map.  Note that $f$ breaks the weakly balanced condition for $-K_{X}$: we have
\begin{equation*}
(a(X,-K_{X}),b(X,-K_{X})) = (1,3) < (1,4) = (a(Y,-f^{*}K_{X}),b(Y,-f^{*}K_{X}))
\end{equation*}
As discussed in Section \ref{absection}, $f$ naturally defines a contraction of faces $F(Y,f^{*}L) \to F(X,L)$ of the nef cone of curves.  Here $F(Y,f^{*}L)$ consists of curves which have vanishing intersection against the blow-up $E$ of the diagonal.  This face has dimension $4$, since it contains the classes of strict transforms of curves on $S \times S$ which do not intersect the diagonal.  Its image $F(X,L)$ is the cone spanned by the classes of the curves $F_{1}(1,2)$ and $F_{2}(1,2)$.  (Here $F_{1}(1,2)$ denotes the curve parametrizing length two subschemes of $S$ where one point is fixed and the other varies in a fiber of the first projection.  $F_{2}(1,2)$ is defined analogously for the second projection.)  Note that $f_{*}$ decreases the dimension of $F(Y,f^{*}L)$ by $2$.

It is easy to see that a dominant component of rational curves on $Y$ with class $\beta \in F(Y,f^{*}L)$ is the strict transform of a dominant component of rational curves on $S \times S$.  Since $S \times S$ is toric, there is exactly one irreducible component of each class $\beta$.  Suppose that $\beta$ is the strict transform of a degree $(a,b,c,d)$ curve class on $(\mathbb{P}^{1})^{\times 4}$.  The pushforward identifies all classes with $a+c=m$ and $b+d=n$ to the class $mF_{1}(1,2) + nF_{2}(1,2)$.  The pushforward of any component of $\mathrm{Mor}(\mathbb{P}^{1},Y)$ with a class in $F(Y,f^{*}L)$ yields (a dense subset of) a component of $\mathrm{Mor}(\mathbb{P}^{1},X)$ since the expected dimensions coincide.  Furthermore, a component of class $mF_{1}(1,2) + nF_{2}(1,2)$ will be the image of exactly two different components on $Y$ (given by $(a,b,m-a,n-d)$ and $(m-a,n-d,a,b)$) except when $m$ and $n$ are both even and $a=c$ and $b=d$, in which case there is only one component.  
Thus there are at least $\lceil \frac{1}{2}(m+1)(n+1) \rceil $ different components of rational curves of class $mF_{1}(1,2) + nF_{2}(1,2)$.  Since any rational curve on $X$ avoiding $E$ is the pushforward of a rational curve on $Y$, this is in fact the exact number.
\end{exam}

In the previous example the growth in components was caused by the existence of a dominant map $f: Y \to X$ breaking the weakly balanced condition.  But even when there is no such map we can still have growth of components due to the existence of face  contracting maps.

\begin{exam} \label{facecontracting2}
Let $X$ be the smooth weak del Pezzo surface in Example \ref{exam:facecontracting}; we retain the notation from this example.  We claim that there are families of free rational curves representing two linearly independent classes in $F(Y,f^{*}L)$.  Using a gluing argument, one can then deduce that the number of components of free rational curves for classes in $F$ grows at least linearly as the degree increases.

For each generator of $F(Y,f^{*}L)$ we can run an MMP to obtain a Mori fibration $\pi: \tilde{Y} \to Z$ on a birational model of $Y$ which contracts this ray.  If $\dim(Z)=1$, then a general fiber will be in the smooth locus of $\tilde{Y}$ and its pullback on $Y$ will be a free rational curve of the desired numerical class.  If $\dim(Z)=0$, we can apply \cite[1.3 Theorem]{KM99} to find a rational curve in the smooth locus of $\tilde{Y}$ whose pullback to $Y$ will be a free rational curve of the desired numerical class.
\end{exam}

As in the previous examples, one can expect the degree of the polynomial $P(d)$ in Conjecture \ref{conj:numberofcomponents} to be controlled by the relative dimension of contracted faces.

\begin{conj} \label{conj:binvandcomponents}
Let $X$ be a smooth projective weak Fano variety.  For a numerical class $\alpha \in \Nef_{1}(X)_{\mathbb{Z}}$, let $h(\alpha)$ denote the number of components of $\mathrm{Mor}(\mathbb P^1, X, \alpha)$ that generically parametrize free curves.  Then $h(m\alpha)$, considered as a function of $m$, is bounded above by a polynomial $P(m)$ whose degree is the largest relative dimension of a map $f_{*}: F(Y,f^{*}L) \to F$ where $f$ is a face contracting morphism $f: Y \to X$, $F$ denotes the image of $F(Y,f^{*}L)$, and $\alpha \in F$.
\end{conj}

\subsection{Breaking chains of free curves}
\label{subsec: breaking chains}

In this section we prove some structure theorems for chains of free curves.  We will pass from working with the spaces $\mathrm{Mor}(\mathbb{P}^{1},X)$ to the Kontsevich spaces of stable maps $\overline{\mathcal{M}}_{0,n}(X,\beta)$.  Note that this change in setting drops the expected dimension of spaces of rational curves by $3$.  We will assume familiarity with these spaces as in \cite{BM96}, \cite{HRS04}.  In fact, we will work exclusively with the projective coarse moduli space $\overline{M}_{0,n}(X,\beta)$.

Let $X$ be a smooth projective uniruled variety and let $L$ be a big and nef $\mathbb{Q}$-divisor on $X$.  By a component of $\overline{M}_{0,0}(X)$, we will mean more precisely the reduced variety underlying some component of this projective scheme.  For each component $M \subset \overline{M}_{0,0}(X)$ which generically parametrizes free curves, we denote by $M'$ the unique component of $\overline{M}_{0,1}(X)$ parametrizing a point on a curve from $M$, and by $M''$ the analogous component of $\overline{M}_{0,2}(X)$.

\begin{defi}
A chain of free curves on $X$ of length $r$ is a stable map $f: C \to X$ such that $C$ is a chain of rational curves with $r$ components and the restriction of $f$ to any component $C_{i}$ realizes $C_{i}$ as a free curve on $X$.
\end{defi}

We can parametrize chains of free curves (coming from components $M_{1},\ldots,M_{r}$) by the product
\begin{equation*}
M'_{1} \times_{X} M''_{2} \times_{X} \ldots \times_{X} M''_{r-1} \times_{X} M'_{r}.
\end{equation*}
Of course such a product might also have components which do not generically parametrize chains of free curves.  We will use the following notation to distinguish between the two types of component.

\begin{defi}
Given a fiber product as above, a ``main component'' of the product is any component which dominates the parameter spaces $M_{i}''$, $M_{1}'$, and $M_{r}'$ under each projection map.
\end{defi}

Loosely speaking our goal is to count such main components.  Note that any component of $M'_{1} \times_{X} \ldots \times_{X} M'_{r}$ which generically parametrizes chains of free curves will have the expected dimension $-K_{X}\cdot C + \dim(X) -2 - r$.  A chain of free curves is automatically a smooth point of $\overline{M}_{0,0}(X)$.

For a component $M_{i}$ which generically parametrizes free curves, we let $U_{i}$ denote the sublocus of free curves.  Analogously, we define $U'_{i}$ and $U''_{i}$ for the one or two pointed versions.

\begin{lemm} \label{lemm: equifibers}
Consider an open component of chains of free curves with a marked point on each end, that is,
\begin{equation*}
N \subset U''_{1} \times_{X} U''_{2} \times_{X} \ldots \times_{X} U''_{r}.
\end{equation*}
Then each projection map $N \to U''_{j}$ is dominant and flat.  Furthermore, the map $N \to X$ induced by the last marked point is dominant and flat.
\end{lemm}

\begin{proof}
The proof is by induction on the length of the chain.  In the base case of one component, the first statement is obvious.  For the second, note that $U''' \subset \overline{M}_{0,3}(X)$ can be identified with an open set in $\mathrm{Mor}(\mathbb P^1, X)$.  By \cite[Corollary 3.5.4]{Kollar} the evaluation map for the second marked point is flat.  This factors through the natural map $U''' \to U''$; since the forgetful map is faithfully flat, we see that the second statement also holds.

We next prove the induction step.  The projection from $N$ onto the first $n-1$ factors maps $N$ to a component $Q$ of $U''_{1} \times_{X} \ldots \times_{X} U''_{r-1}$.  By induction, $Q$ has the two desired properties.  Also, since the space of free curves through a fixed point has the expected dimension, the map $U''_{r} \to X$ induced by the first marked point is dominant and flat.  Consider the diagram
\begin{equation*}
\xymatrix{Q \times_{X} U''_{r} \ar[r] \ar[d] &  U''_{r} \ar[d]\\
Q \ar[r] & X}
\end{equation*}
Both projections from $Q \times_{X} U''_{r}$ have equidimensional fibers by base change.  Furthermore, every component of $Q \times_{X} U''_{r}$ has the same dimension (as it parametrizes chains of free curves).  Together, this shows that every component of $Q \times_{X} U''_{r}$ will dominate $Q$ so long as there is some component which dominates $Q$.  But it is clear that a general chain of free curves in $Q$ can be attached to a free curve in $U_{r}$, so that the map from $Q \times_{X} U''_{r}$ to $Q$ must be dominant for at least one component.  Noting that $N$ is a component of $Q \times_{X} U''_{r}$ for dimension reasons, we obtain from the induction hypothesis the first statement for $N$ since flatness is stable under base change and composition.  The last statement follows by the same logic.
\end{proof}

\begin{lemm} \label{lemm: familiesthroughpoints}
Consider a component of chains of free curves with a marked point on each end, that is, a main component
\begin{equation*}
N \subset M''_{1} \times_{X} M''_{2} \times_{X} \ldots \times_{X} M''_{r}.
\end{equation*}
Fix a closed subset $Z \subsetneq X$.  Consider the map $f: N \to X$ induced by the first marked point.  For the fiber $F$ of $f$ over a general point of $X$, every component of $F$ generically parametrizes a chain of free curves $C$ such that the map $g: C \to X$ induced by the last marked point does not have image in $Z$.
\end{lemm}

\begin{proof}
The proof is by induction on the length of the chain.  Consider the base case $f: M''_{1} \to X$.  Since reducible curves form a codimension $1$ locus, every component of a general fiber of $f$ must contain irreducible curves.  Since there is a closed subset of $X$ containing every non-free irreducible curve in $M_{1}$, we see that every component of a general fiber of $F$ must contain free curves.  The ability to avoid $Z$ follows from Lemma \ref{lemm: equifibers}.

We now prove the induction step.  Via projection $N$ maps into a component $Q \subset M''_{1} \times_{X} \ldots \times_{X} M''_{r-1}$.  By induction $Q$ satisfies the desired property.  
By \cite[II.3.5.4 Corollary and II.3.10.1 Corollary]{Kollar} there is a proper closed subset $Z_{0} \subset X$ such that if $C_{0}$ is a component of a curve parametrized by $M''_{r}$ and $C_{0} \not \subset Z_{0}$ then $C_{0}$ is free.  Consider the evaluation along the first marked point (of $Q$) denoted by $\tilde{f}: Q \times_{X} M''_{r} \to X$.  The fibers of this map are products $F \times_{X} M''_{r}$ where $F$ is a fiber of $Q \to X$; choosing the fiber $F$ general with respect to $Z_{0}$, we see that every component of every fiber will contain a chain of free curves.  In particular, this is also true for the map $f: N \to X$ which is a restriction of $\tilde{f}$ to a component.  The ability to avoid $Z$ via the last marked point follows from Lemma \ref{lemm: equifibers}.
\end{proof}

\begin{lemm}
\label{lemm: chains}
Consider a parameter space of chains of free curves, that is, a main component
\begin{equation*}
N \subset M'_{1} \times_{X} M''_{2} \times_{X} \ldots \times_{X} M''_{r-1} \times_{X} M'_{r}.
\end{equation*}
Suppose that curves in $M_{j}$ degenerate into a chain of two free curves in $\widetilde{M}_{j}' \times_{X} \widehat{M}_{j}'$.  Then $N$ contains a main component of 
\begin{equation*}
M'_{1} \times_{X} M''_{2} \times_{X} \ldots \times_{X} M''_{j-1} \times_{X} \widetilde{M}_{j}'' \times_{X} \widehat{M}_{j}'' \times_{X} M_{j+1}'' \times_{X} \ldots \times_{X} M''_{r-1} \times_{X} M'_{r}.
\end{equation*}
\end{lemm}

\begin{proof}
By definition the projection $N \to M''_{j}$ is dominant, hence surjective by properness.  So we know that $N$ contains a point of
\begin{equation*}
M'_{1} \times_{X} M''_{2} \times_{X} \ldots \times_{X} M''_{j-1} \times_{X} \widetilde{M}_{j}'' \times_{X} \widehat{M}_{j}'' \times_{X} M_{j+1}'' \times_{X} \ldots \times_{X} M''_{r-1} \times_{X} M'_{r}.
\end{equation*}
If we can show that it contains a point which is a chain of free curves, then since such points are smooth in $\overline{M}_{0,0}(X)$ we can conclude that $N$ will contain an entire component of chains of length $r+1$.

Since the map $N \to M''_{j}$ is surjective, in particular, for any two-pointed length $2$ chain in $\widetilde{M}_{j}'' \times_{X} \widehat{M}_{j}''$ there is a curve parametrized by $N$ containing this chain.  Since the curves are free, we may choose a chain such that the first and last marked points are general.  The fiber of $N$ over this point is a union of components of
\begin{equation*}
G_{1} \subset M'_{1} \times_{X} \ldots \times_{X} M''_{j-1}
\end{equation*}
under a product with
\begin{equation*}
G_{2} \subset M''_{j+1} \times_{X} \ldots \times_{X} M'_{r}
\end{equation*}
where $G_{1}$ and $G_{2}$ are the fibers of the last or first marking respectively.  Applying Lemma \ref{lemm: familiesthroughpoints}, we see that every component of the fiber over this point contains chains of free curves.
\end{proof}

\subsection{Toward Batyrev's conjecture}

\begin{defi}
Let $C_{1} \cup \ldots \cup C_{r}$ be a chain of free curves, with map $f: C \to X$.  Let $f^{\dagger}: \{ 1,\ldots,r\} \to \overline{M}_{0,0}(X)$ denote the function which assigns to $i$ the unique component of the moduli space containing $C_{i}$.  We call $f^{\dagger}$ the combinatorial type of $f$.
\end{defi}

\begin{lemm} \label{lemm:reorderingcomponents}
Let $X$ be a smooth projective variety and let $M$ be a component of $\overline{M}_{0,0}(X)$.  Suppose that $M$ contains a point $f$ parametrizing a chain of free curves.  For any $\tilde{f}^{\dagger}$ which is a precomposition of $f^{\dagger}$ with a permutation, $M$ also contains a point representing a chain of free curves with combinatorial type $\tilde{f}^{\dagger}$.
\end{lemm}

\begin{proof}
Suppose that $f: C \to X$ denotes our original chain of free curves.  It suffices to prove the statement when $\tilde{f}^{\dagger}$ differs from $f^{\dagger}$ by a transposition of two adjacent elements.  Suppose that $T_{1}$ and $T_{2}$ are two adjacent components of $C$.  Let $S_{1}$ denote the rest of the chain which attaches to $T_{1}$, and $S_{2}$ denote the rest of the chain which attaches to $T_{2}$.  After deforming $f$, we may suppose that the intersection of $T_{1}$ and $T_{2}$ maps to a general point $x$ of $X$.

Suppose we leave $T_{1}$ and $T_{2}$ fixed, but deform $S_{1}$ and $S_{2}$, maintaining a point of intersection with $T_{1}$ or $T_{2}$ respectively, so that the specialized curves $S_{1}'$ and $S_{2}'$ contain the point $x$.  By generality of the situation, the deformed $S_{1}'$ and $S_{2}'$ are still chains of free curves.  The stable curve $g: D \to X$ corresponding to this deformation looks like a single rational curve $Z$ contracted by $g$ to the point $x$ with four chains of free curves $S'_{1}, T_{1}, T_{2}, S'_{2}$ attached to $Z$.  A tangent space calculation (as in \cite[Corollary 1.6]{Testa05}) shows that $g$ is a smooth point of $M$.  However, by a similar argument $g$ is the deformation of a chain of rational curves of the type $S_{1}'' \cup T_{2} \cup T_{1} \cup S_{2}''$ where $S_{1}''$ and $S_{2}''$ are deformations of $S_{1}'$ and $S_{2}'$.
\end{proof}

\begin{lemm} \label{lemm:finitenessinsmalldegree}
Let $X$ be a smooth projective variety and let $L$ be a big and nef $\mathbb{Q}$-divisor on $X$.  Fix a positive integer $q$.  Consider the set $\mathcal{Z}$ of generically finite dominant covers $f: Z \to X$ such that there is a component $T$ of $\overline{M}_{0,0}(Z)$ which generically parametrizes free curves of $f^{*}L$-degree $\leq q$ and such that the induced map $T \to \overline{M}_{0,0}(X)$ is dominant birational onto a component.  
Up to birational equivalence, there are only finitely many elements of $\mathcal{Z}$.
\end{lemm}

\begin{proof}
For degree reasons, there are only finitely many components $M$ of $\overline{M}_{0,0}(X)$ which can be the closure of the image of such a map.  Each such $M$ generically parametrizes free curves.  Thus, there is a unique component $M'$ of $\overline{M}_{0,1}$ lying over $M$.  Let $g: M' \to X$ denote the universal family map, and let $h: \overline{Z} \to X$ denote the Stein factorization of a resolution of $g$.  If the map $g$ factors rationally through a generically finite dominant map $f: Z \to X$, then so does $h$.  Thus for any given component $M$ there can only be finitely many corresponding elements of $\mathcal{Z}$.
\end{proof}

\begin{theo} \label{theo:polybounds}
Let $X$ be a smooth projective variety and let $L$ be a big and nef $\mathbb{Q}$-divisor on $X$.  Fix a positive integer $q$ and fix a subset $\mathcal{N} \subset \overline{M}_{0,0}(X)$ where each component generically parametrizes free curves of $L$-degree at most $q$.  There is a polynomial $P(d)$ such that there are at most $P(d)$ components of $\overline{M}_{0,0}(X,d)$ which contain a chain of free curves of total $L$-degree $d$ where each free curve is parametrized by a component of $\overline{M}_{0,0}(X)$ contained in $\mathcal{N}$.
\end{theo}

Note that any chain of free curves on $X$ can be smoothed yielding a free curve (\cite[II.7.6 Theorem]{Kollar}).  Thus any component of $\overline{M}_{0,0}(X,d)$ which contains a chain of free curves must generically parametrize free curves. 

\begin{proof}
Let $\{ M_{\beta} \}_{\beta \in \Xi}$ denote the elements of $\mathcal{N}$.  For each component, we have a universal family map $\nu_{\beta}: M_{\beta}' \to X$.  We let $e_{\beta}$ denote the degree of the Stein factorization of the composition of $\nu_{\beta}$ with a resolution of singularities of $M_{\beta}'$ and set $e = \sup_{\beta \in \Xi} e_{\beta}$.  Note that since we have included a resolution in the definition, if $e_{\beta} = 1$ then the general fiber of $\nu_{\beta}$ is irreducible.

The proof is by induction on $e$.  First suppose that $e=1$.  Let $M$ be a component of $\overline{M}_{0,0}(X,d)$ satisfying the desired condition.  Since a chain of free curves is a smooth point of $M$, to count such components $M$ it suffices to count all possible components of the parameter space of chains of free curves from $\mathcal{N}$ of total degree $d$.  In fact, by applying Lemma \ref{lemm:reorderingcomponents}, we may reorder the combinatorial type however we please.  Furthermore, for any choice of combinatorial type the parameter space of chains of that type
\begin{equation*}
M'_{1} \times_{X} M''_{2} \times_{X} \ldots \times_{X} M''_{r-1} \times_{X} M'_{r}
\end{equation*}
is irreducible since by assumption each degree $e_{i}=1$.  Thus the number of possible $M$ is at most the possible ways of choosing (with replacement and unordered) components of $\mathcal{N}$ such that the total degree adds up to $d$.  This count is polynomial in $d$.

Before continuing with the proof, we make an observation:

\begin{obse} \label{obse:irreducibility}
Suppose that $M_{1}$ and $M_{2}$ are components of $\overline{M}_{0,0}(X)$ which generically parametrize free curves, and that (a resolution of) the map $M'_{1} \to X$ has degree $1$.  Then there is a unique main component of $M'_{1} \times_{X} M'_{2}$ which parametrizes length $2$ chains of free curves.  Indeed, since the general fiber of $M'_{1} \times_{X} M'_{2} \to M'_{2}$ is irreducible and $M'_{2}$ is irreducible, we see that $M_1' \times_X M_2'$ is irreducible.
\end{obse}


Now suppose that $e > 1$.  For a dominant generically finite map $g: Z \to X$ of degree $\geq 2$ with $Z$ smooth, let $\mathcal{N}_{Z}$ denote the subset of $\mathcal{N}$ consisting of components $M$ such that the universal map $M' \to X$ factors rationally through $Z$.  Note that the locus where the rational map to $Z$ is not defined must miss the general fiber of $M' \to M$.  Thus, we obtain a family of free curves on $Z$ parametrized by an open subset of $M$.  A deformation calculation shows that the general curve has vanishing intersection with the ramification divisor; in particular, $M$ is birational to a component of $\overline{M}_{0,0}(Z)$.  Note that for any component of $\mathcal{N}_{Z}$ the degree of the rational map from the component to $Z$ is strictly smaller than for the corresponding component in $\mathcal{N}$.  Consider the corresponding families of rational curves on $Z$ measured with respect to the big and nef divisor $g^{*}L$.  By the induction hypothesis, there is a polynomial $P_{Z}(d)$ which gives an upper bound for the number of components of $\overline{M}_{0,0}(Z,d)$ which arise by gluing chains from $\mathcal{N}_{Z}$ on $Z$.  Furthermore, by Lemma \ref{lemm:finitenessinsmalldegree} there are only finitely many $Z$ for which $\mathcal{N}_{Z}$ is non-empty.  


Fix a positive integer $r$ and a dominant generically finite map $f: Z \to X$ of degree $\geq 2$ with $Z$ smooth.  As we vary over possible choices $M_{i} \in \mathcal{N}$, consider all main components of $M'_{1} \times_{X} \ldots \times_{X} M'_{k}$ such that there is an integer $b$ where the component $M$ of $\overline{M}_{0,0}(X)$ obtained by gluing the first $b$ curves has $L$-degree $r$ and has a universal family map which factors rationally through $Z$, but if we consider the component arising from gluing the first $b+1$ curves, the Stein factorization of a resolution of the universal family map has degree $1$. 
We see there are at most $P_{Z}(r)$ possible components $M$ obtained by gluing the first $b$ curves in the chain.
Next consider adding one more component.  By degree considerations, there can be at most $e \cdot P_{Z}(r)$ components obtained by gluing the first $(b+1)$ curves, and for any such component the universal family has map to $X$ with generically irreducible fibers.  Finally, to add on the remaining components, we may use Lemma \ref{lemm:reorderingcomponents} to reorder the other components arbitrarily.  Applying Observation \ref{obse:irreducibility}, we see that the total number of glued components for this choice of $b$ and $Z$ is bounded above by $e \cdot P_{Z}(r)$ times the number of ways to choose (with replacement and unordered) $(k-b-1)$ components from $\mathcal{N}$.

In total, the number of components of $\overline{M}_{0,0}(X)$ containing chains of curves from $\mathcal{N}$ of degree $d$ will be bounded above by the sum of the previous bounds as we vary $Z$ and $r$.  Let $Q(k)$ denote the polynomial representing the number of ways to choose $k$ components (with replacement and unordered) from $\mathcal{N}$.  Altogether, the number of components is bounded above by the polynomial in $d$ given by
\begin{equation*}
e \cdot Q(d) \cdot \sum_{Z} \sum_{r \leq d} P_{Z}(r).
\end{equation*}
\end{proof}

\subsection{Gluing free curves}

In this section we attempt to improve the degree of the polynomial bound constructed in Theorem \ref{theo:polybounds}.  Returning to the proof, we see that the degree of the Stein factorization of $\mathcal{C} \to X$ (where $\mathcal{C}$ is a universal family of rational curves) plays an important role.  The key observation is that we can use the $a$-invariant to control the properties of this Stein factorization.

\begin{prop} \label{prop:highdegreemeansequala}
Let $X$ be a smooth projective weak Fano variety.  Suppose that $W$ is a component of $\mathrm{Mor}(\mathbb P^1, X)$ parametrizing a dominant family of rational curves $\pi: \mathcal C \rightarrow W$ with evaluation map $s: \mathcal C \rightarrow X$.  Let $\widetilde{\mathcal C}$ be the resolution of a projective compactification of $\mathcal C$ with a morphism $s': \widetilde{\mathcal C} \rightarrow X$ extending the evaluation map. Consider the Stein factorization of $s'$, $\widetilde{\mathcal C} \rightarrow Y \rightarrow^{f} X$.  Then $a(Y,-f^{*}K_X) = a(X,-K_{X})$.
\end{prop}

\begin{proof}
Let $\widetilde{Y}$ be a resolution of $Y$ with map $\widetilde{f}: \widetilde{Y} \to X$.  By taking the strict transform of the family of rational curves, one obtains a dominant family on $\widetilde{Y}$ which is parametrized by an open subset of a component of $\mathrm{Mor}(\mathbb P^1, \widetilde{Y})$.  Since the dimension of this component is the same on $\widetilde{Y}$ and on $X$, and equals to the expected dimension in both cases, we have
\[
K_X \cdot C = K_{\widetilde{Y}} \cdot C \qquad \implies \qquad (K_{\widetilde{Y}} - \widetilde{f}^{*}K_{X}) \cdot C = 0.
\]
Thus the divisor $K_{\widetilde{Y}} + a(X,-K_{X})(-\widetilde{f}^{*}K_{X})$ is pseudo-effective but not big.
\end{proof}

Suppose now that $X$ is a smooth projective weak Fano variety satisfying:
\begin{itemize}
\item $X$ does not admit any $a$-cover.
\item every free curve on $X$ deforms (as a stable map) to a chain of free curves of degree $\leq q$.
\end{itemize}
Since the first condition holds, we can apply Proposition \ref{prop:highdegreemeansequala} to see that for every component of $\mathrm{Mor}(\mathbb P^1, X)$ parametrizing a dominant family of rational curves the evaluation morphism has connected fibers.  Since the second condition holds, we can apply Theorem \ref{theo:polybounds} to control the number of components of the parameter space of rational curves.  

Let $\mathcal{S}$ be the set of components of $\mathrm{Mor}(\mathbb{P}^{1},X)$ that generically parametrize free curves of degree $\leq q$.  Consider the abelian group $\Lambda = \oplus_{M \in \mathcal S}\mathbb{Z}M$.  For sequences $\{M_{i}\}_{i=1}^{s}$, $\{M_{j}'\}_{j=1}^{t}$ of elements in $\mathcal{S}$ we introduce the relation $\sum M_{i} = \sum M_{j}'$ whenever a chain of free curves parametrized by the $M_{i}$ lies in the same component of $\overline{M}_{0,0}(X)$ as a chain of free curves which lie in the $\{M_{j}'\}$.  The argument of Theorem \ref{theo:polybounds} shows that the total number of components of $\Mor(\mathbb{P}^{1},X)$ parametrizing free curves of degree $\leq m$ is bounded above by a polynomial in $m$ of degree
\begin{equation*}
\mathrm{rank}(\Lambda/ R)
\end{equation*}
where $R$ is the set of relations described above.  By analyzing components of $\Mor(\mathbb{P}^{1},X)$ of low degree, one can hope to obtain enough relations to verify Conjecture \ref{conj:binvandcomponents}.  For example:

\begin{coro}
Let $X$ be a smooth projective Fano variety of Picard rank $1$ satisfying:
\begin{itemize}
\item $X$ does not admit any $a$-cover.
\item every free curve on $X$ deforms (as a stable map) to a chain of free curves of degree $\leq q$.
\end{itemize}
Suppose that the space of free curves of degree $q!$ is irreducible.  Then there is an upper bound on the number of components of $\Mor(\mathbb{P}^{1},X,\alpha)$ parametrizing free curves as we vary the class $\alpha \in \Nef_{1}(X)_{\mathbb{Z}}$.
\end{coro}


\section{Geometric Manin's Conjecture} \label{section:statementofmc} 

In this section we present a precise version of Manin's Conjecture for rational curves.   We will need the following definitions:

\begin{defi} \label{defi:alphaconstant}
Let $X$ be a smooth uniruled projective variety and let $L$ be a big and nef $\mathbb{Q}$-divisor on $X$.
\begin{itemize}
\item The rationality index $r(X,L)$ is the smallest positive rational number of the form $L \cdot \alpha$ as $\alpha$ varies over all classes in $N_{1}(X)_{\mathbb{Z}}$.
\item  Let $V$ be the subspace of $N_{1}(X)$ spanned by $F(X,L)$.  (Note that by \cite[Theorem 2.16]{HTT15} $V$ is a rational subspace with respect to the lattice of curve classes.)  Let $Q$ denote the rational hyperplane in $V$ consisting of all curve classes with vanishing intersection against $L$; there is a unique measure $d\Omega$ on $Q$ normalized by the lattice of integral curve classes.  This also induces a measure on the parallel affine plane $Q_{r} := \{ \beta \in V | L \cdot \beta = r(X,L) \}$.  We define $\alpha(X,L)$ to be the volume of the polytope $Q_{r} \cap F(X,L)$.

In other words, $\alpha(X,L)$ is the top coefficient of the Ehrhart polynomial for the polytope obtained by slicing $F$ by the codimension $1$ plane $Q_{r}$.
\end{itemize}
\end{defi}

\subsection{Statement of conjecture: rigid case}  \label{section:rigidcase}
Manin's Conjecture predicts the growth rate of components of $\mathrm{Mor}(\mathbb{P}^{1},X)$ after removing the rational curves in some ``exceptional set.''  In the number-theoretic setting, removing points from a closed subset is not sufficient to obtain the expected growth rate; one must remove a thin set of points (see \cite{BT-cubic},  \cite{Peyre03}, \cite{BL16}, \cite{LeRudulier}).   Following the results of \cite{LT16}, we will interpret a ``thin set of rational curves'' via the geometry of the $a$ and $b$ constants.

In this section we will address the situation when $\kappa(K_{X} + a(X,L)L)=0$.  Note that this includes the case when $X$ is weak Fano and $L = -K_{X}$.  The following definition identifies exactly which components should be counted in this situation; it is identical to the conjectural description of the exceptional set for rational points.

\begin{defi} \label{defi:manincomponent}
Let $X$ be a smooth projective uniruled variety and let $L$ be a big and nef $\mathbb{Q}$-divisor on $X$ such that $\kappa(K_{X} + a(X,L)L) = 0$.  Let $M \subset \Mor(\mathbb{P}^{1},X)$ be a component, let $\mathcal{C}$ denote the universal family over $M$ and let $s: \mathcal{C} \to X$ denote the family map.  We say that $M$ is a Manin component if:
\begin{enumerate}
\item The curves parametrized by $M$ have class contained in $F(X,L)$.
\item The morphism $s$ does not factor rationally through any thin morphism $f: Y \to X$ such that $a(Y,f^{*}L) > a(X,L)$.
\item The morphism $s$ does not factor rationally through any dominant thin morphism $f: Y \to X$ such that $f$ is face contracting and
\begin{equation*}
(a(Y,f^{*}L),b(Y,f^{*}L)) \geq (a(X,L),b(X,L)) 
\end{equation*}
in the lexicographic order.
\item The morphism $s$ does not factor rationally through any dominant thin morphism $f: Y \to X$ such that $a(Y, f^*L) = a(X, L)$ and $\kappa(K_{Y} + a(Y,f^*L)f^{*}L) > 0$.
\end{enumerate}
\end{defi}

Note that by Theorem \ref{theo:fanocase} any Manin component will necessarily parametrize a dominant family of curves.

\begin{rema}
Condition (ii) is necessitated by Theorem \ref{theo:closednotfree} and condition (iii) is motivated by Conjecture \ref{conj:binvandcomponents}, but we have not yet discussed condition (iv).  For rational points, such a restriction is necessary to obtain the correct Peyre's constant; see \cite{BL16}.  For rational curves, this condition rules out ``extraneous'' components consisting of curves that are free but not very free.  Again such components can modify the leading constant in Manin's Conjecture; see Theorem \ref{theo:componentsforfanothreefolds} for an example.  In order to obtain uniqueness in Conjecture \ref{conj:componentsandabreaking} below one must include condition (iv).
\end{rema}

\begin{rema}
Proposition \ref{prop:highdegreemeansequala} and \cite[Proposition 4]{Kollar15} show that the curves parametrized by Manin components will almost always satisfy the weak Lefschetz property.
\end{rema}

Our main conjecture concerning Manin components is:

\begin{conj}[Strong Manin's Conjecture] \label{conj:componentsandabreaking}
Let $X$ be a smooth projective uniruled variety and let $L$ be a big and nef $\mathbb{Q}$-divisor such that $\kappa(K_{X} + a(X,L)L) = 0$.  For any $\mathbb{Z}$-curve class $\alpha$ contained in the relative interior of $F(X,L)$ there is at most one Manin component parametrizing curves of class $\alpha$.
\end{conj}

To obtain the correct growth rate it would be enough to show that the number of Manin components representing a numerical class is bounded above, but uniqueness holds in every example we know about.  Since Conjecture \ref{conj:componentsandabreaking} is quite strong, we will formulate a weaker version which emphasizes the relationship with the theory of rational points.  Define the counting function
$$N(X, L, q, d) = \sum_{i=1}^{d} \sum_{W \in \mathrm{Manin}_{i}} q^{\dim W}$$
where $\mathrm{Manin}_{i}$ is the set of Manin components of $\mathrm{Mor}(\mathbb P^1, X)$ parametrizing curves of $L$-degree $i r(X,-K_{X})$.  

\begin{conj}[Geometric Manin's Conjecture]
Let $X$ be a smooth projective uniruled variety of dimension $n$ and let $L$ be a big and nef $\mathbb{Q}$-divisor on $X$ such that $\kappa(K_{X} + a(X,L)L) = 0$.  
Then
\begin{equation*}
N(X, L, q, d) \sim \frac{q^{\dim X}\alpha(X,L)}{1-q^{-a(X,L)r(X,L)}} q^{da(X,L)r(X,L)}d^{b(X,L)-1}.
\end{equation*}
\end{conj}

\begin{rema}
Suppose that $X$ is a smooth projective uniruled variety and that $L$ is a big and nef $\mathbb{Q}$-divisor such that $\kappa(K_{X} + a(X,L)L)=0$.  Loosely speaking, we expect a bijection between Manin components on $X$ and components of families of rational curves on some birational model $X'$ of $X$.  In other words, our counting function should actually count curves on some variety and not just curves in some face.

More precisely, the argument of \cite[Theorem 3.5]{LT16} shows that there is a birational model $\phi: X \dashrightarrow X'$ such that $\phi$ is a rational contraction, $X'$ is normal $\mathbb{Q}$-factorial with terminal singularities, and the anticanonical divisor on $X'$ is big and nef and satisfies $-K_{X'} \equiv \phi_{*}a(X,L)L$.  We expect Manin components to correspond to components of the moduli space of rational curves on $X'$.

Here is a heuristic argument.  Fix a class $\alpha \in F(X,L)$.  Suppose that $\beta \in \Nef_{1}(X)_{\mathbb{Z}}\backslash F(X,L)$ is a curve class whose pushforward to the model $X'$ is the same as $\alpha$.  Families of rational curves of class $\beta$ have lower expected dimension than families of rational curves of class $\alpha$, so the former should form subfamilies of the latter under pushforward.  In particular, families of curves of class $\beta$ should not contribute to the count of families of rational curves on $X'$.  Thus counting families of rational curves on $X'$ should be the same as counting families of rational curves contained in $F(X,L)$.

This heuristic anticipates the existence of many families of rational curves with vanishing intersection against $K_{X} + a(X,L)L$.  But  the existence of even a single such family is a famous open problem in birational geometry (see Conjecture \ref{conj: rational}).
\end{rema}

Conjecture \ref{conj:componentsandabreaking} is known for the following Fano varieties (equipped with the anticanonical polarization):
\begin{itemize}
\item general hypersurfaces in $\mathbb{P}^{n}$ of degree $<n-1$ by \cite{RY16},
\item homogeneous varieties by \cite{Thomsen98}, \cite{KP01},
\item Fano toric varieties by work of Bourqui, e.g.~\cite{Bou16},
\item Del Pezzo surfaces by \cite{Testa09}.
\end{itemize}
In the last two cases we need to explain how to derive the result from the cited papers.

\begin{exam} \label{exam:delpezzosurface}
Let $X$ be a smooth del Pezzo surface of degree $\geq 2$.  Fix a nef curve class $\alpha$ and consider the space parametrizing dominant families of rational curves of class $\alpha$.  
For simplicity we may assume that $X$ has index $1$, i.e., it contains a $(-1)$-curve.
Then:
\begin{itemize}
\item \cite{Testa09} shows that the sublocus parametrizing maps birational onto their image is either irreducible or empty.
\item An easy deformation count shows that if there is a component parametrizing maps which are non-birational onto their image, the image must be a fiber of a map from $X$ to $\mathbb{P}^{1}$.
\end{itemize}

\cite[Theorem 6.2]{LT16} classifies the behavior of $a$ and $b$ constants for subvarieties and covers of del Pezzo surfaces.  It shows that:
\begin{itemize}
\item The only curves $C$ with $a(C,-K_{X}|_{C}) > a(X,-K_{X})$ are $(-1)$-curves.
\item There are no dominant thin maps $f: Y \to X$ such that $a(Y,-f^{*}K_{X}) = a(X,-K_{X})$ and $\kappa(K_{Y} - a(X,-K_{X})f^{*}K_{X}) = 0$.
\item Suppose $f: Y \to X$ is a dominant thin map such that $a(Y,-f^{*}K_{X}) = a(X,-K_{X})$ and $\kappa(K_{Y} - a(X,-K_{X})f^{*}K_{X}) = 1$.  The fibers of the Iitaka fibration for $Y$ are mapped under $f$ to the fibers of a map from $X$ to a curve.
\end{itemize}
Based on this analysis, Conjecture \ref{conj:componentsandabreaking} is verified by Testa's results.
\end{exam}

\begin{exam}  
Let $X$ be a smooth projective toric variety with open torus $U$.  \cite[Theorem 1.10]{Bou16} shows that every nef curve class which has intersection $\geq 1$ against every torus-invariant divisor is represented by a unique dominant family of rational curves. 
\cite[Example 8.3]{LT18} analyzes the behavior of the $a$-invariant for subvarieties and covers of toric varieties.  Based on this analysis, Conjecture \ref{conj:componentsandabreaking} is verified by Bourqui's results, i.e., the unique component representing a nef class with the above property is a Manin component.
\end{exam}

\subsection{Outline of conjecture: general case}
The formulation of Manin's Conjecture in the general case should be essentially the same.  Let $X$ be a smooth projective uniruled variety and let $L$ be a big and nef $\mathbb{Q}$-divisor such that $\kappa(K_{X} + a(X,L)L) > 0$.  After replacing $X$ by a birational model, we can assume that the Iitaka fibration for $K_{X} + a(X,L)L$ is a morphism $\pi$.  The definition of a Manin component is now a bit more subtle; one can no longer focus on dominant maps but must also account for covers of fibers of $\pi$.  However, after making this minor change, Conjecture \ref{conj:componentsandabreaking} and the behavior of the counting function $N(X,L,q,d)$ should be formulated in exactly the same way.

\begin{rema}
In the general case, Manin components should be in bijection with families of rational curves on a birational model of a fiber of the Iitaka fibration of $K_{X} + a(X,L)L$.
\end{rema}

\subsection{Manin-type bounds}

\begin{theo}
Let $X$ be a smooth projective uniruled variety and let $L$ be a big and nef $\mathbb{Q}$-divisor on $X$.  Fix $\epsilon > 0$; then for sufficiently large $q$
\begin{equation*}
N(X,L,q,d) = O \left( q^{d(a(X,L)r(X, L)+\epsilon)} \right).
\end{equation*}
\end{theo}

\begin{proof}
\cite[Equation 0.5]{Manin95} shows that there is a positive constant $C$ such that the number of components of free curves of degree $d$ against a fixed big and nef divisor is at most $C^{d}$.  The result follows by combining this equation with Theorem \ref{theo:closednotfree} and a standard counting argument.
\end{proof}

It is also interesting to look for lower bounds on the number of components of rational curves.  It is conjectured that free rational curves generate the nef cone of curves -- this would follow from the existence of rational curves in the smooth locus of mildly singular Fano varieties.  \cite{TZ14} proves a weaker statement:

\begin{lemm} \label{lemm:freespan}
Let $X$ be a smooth projective rationally connected variety.  Then $N_{1}(X)$ is spanned by the classes of free rational curves. 
\end{lemm}

\begin{proof}
By \cite[Theorem 1.3]{TZ14}, $N_{1}(X)_{\mathbb{Z}}$ is spanned by the classes of rational curves $\{ C_{i} \}$.  Since $X$ is rationally connected, there is a family of very free curves $C$ such that there is a very free member of the family through any point of $X$.  By gluing sufficiently many of these $C$ onto one of the $C_{i}$ to form a comb, we can deform to get a smooth curve (as in \cite[II.7.10 Proposition]{Kollar}).  It is then clear that these smoothed curves and the class of $C$ together span $N_{1}(X)$.
\end{proof}

Free curves which meet at a point can be glued to a free curve of larger degree (see \cite[II.7.6 Theorem]{Kollar}).  Thus one can generate many more dominant components of $\Mor(\mathbb{P}^{1},X)$ starting from this spanning set.  Suppose now that $X$ is a Fano variety and that $L = -K_{X}$.  If all the components of rational curves constructed by gluing are Manin components, then we obtain a lower bound of the form
\begin{equation*}
N(X,-K_{X},q,d) \geq Cq^{dr(X, -K_X)}d^{\rho(X)-1}.
\end{equation*}
for some constant $C$.  However, in general there is no reason for this construction to yield only Manin components.

\subsection{Geometric heuristics} 
In our interpretation of Manin's Conjecture one should discount contributions of $f: Y \to X$ with higher $a$ and $b$-values.  In this section, we give a heuristic argument proving that such components \emph{must} be discounted.  Since we are only interested in heuristics, in this subsection we will assume the following difficult conjecture about rational curves.

\begin{conj}
\label{conj: rational}
Let $X$ be a smooth projective variety and let $L$ be a big and nef $\mathbb{Q}$-divisor on $X$.
For each element $\alpha \in \Nef_{1}(X)_{\mathbb{Z}}$ satisfying $(K_{X} + a(X,L)L) \cdot \alpha = 0$ and with sufficiently high $L$-degree, there exists a dominant family of maps from $\mathbb{P}^{1}$ to $X$ whose images have class $\alpha$.
\end{conj}

Conjecture \ref{conj: rational} would follow quickly from standard conjectures predicting the existence of free rational curves contained in the smooth locus of a log Fano variety.  Assuming this conjecture, the following two statements show that thin morphisms $f: Y \to X$ such that $Y$ has higher $a,b$-values would give contributions to the counting function which are higher than the predicted growth rate. 

\begin{prop}
Assume Conjecture~\ref{conj: rational}.
Let $X$ be a smooth projective weak Fano variety and set $L = -K_{X}$.
Suppose that $f: Y \rightarrow X$ is a generically finite morphism such that $a(Y, L)> a(X,L)$.  Then there exists components of $\mathrm{Mor}(\mathbb{P}^{1},X)$ of any sufficiently high degree which factor through $f(Y)\subset X$ and have higher than the expected dimension.
\end{prop}

\begin{proof}
Choose a dominant family of rational curves $C$ on $Y$ as in Conjecture~\ref{conj: rational} such that
\begin{equation*}
L \cdot C > \frac{\dim(X) - \dim(Y)}{a(Y,L)- a(X,L)}.
\end{equation*}
By computing the expected dimension on $X$ and on $Y$ one concludes the statement.
\end{proof}

\begin{prop} \label{prop:numbercompandbvalues}
Assume Conjecture~\ref{conj: rational}.
Let $X$ be a smooth projective weak Fano variety and set $L = -K_X$.
Suppose that we have a surjective generically finite map $f: Y \rightarrow X$ which is face contracting for $L$.  There is a class $\alpha \in \mathrm{Nef}_1(X)_{\mathbb{Z}}$ such that the number of components of $\mathrm{Mor}(\mathbb P^1, X, m\alpha)$ is bounded below by a polynomial of degree in $m$ equal to the relative dimension of the faces.
\end{prop}

\begin{proof}
Let $b$ denote the difference in dimensions between $F(Y,f^{*}L)$ and its image under $f_{*}$.  \cite[Theorem 2.16]{HTT15} shows that that lattice points $F(Y,f^{*}L) \cap N_{1}(Y)_{\mathbb{Z}}$ generate a subcone of $F(Y,f^{*}L)$ which is full rank.   Thus there is a class $\alpha \in \Nef_{1}(X)_{\mathbb{Z}}$ and a constant $C>0$ such that for sufficiently large integers $m$ there are $\geq Cm^{b} $ points of $F(Y,f^{*}L)$ mapping to $m\alpha$.

For sufficiently large $m$, Conjecture \ref{conj: rational} guarantees that for each class $\beta$ that pushes forward to $m\alpha$ there is a component $M_{\beta} \subset \mathrm{Mor}(\mathbb P^1, Y, \beta)$ parametrizing a dominant family of rational curves.  For each such $M_{\beta}$, by composing with $f$ we get a dominant family of rational curves on $X$.  At most $\deg f$ different components on $Y$ can get identified to a single component on $X$, showing that the number of different components on $X$ has the desired asymptotic growth rate.
\end{proof}

We note in passing that if Conjecture \ref{conj: rational} is true it would allow one to use facts about rational curves to deduce results about rational points.

\begin{exam}
Assuming Conjecture \ref{conj: rational}, the results of \cite{HRS04} show that a general hypersurface with degree not too large will not admit subvarieties with higher $a$-values.  Switching to the number-theoretic setting, we should then expect Manin's Conjecture to hold for such hypersurfaces with no exceptional set. Indeed, such results are obtained in the seminal work \cite{Bir61} using the circle method  when the dimension is exponentially larger than its degree.

Conversely, \cite{BV16} uses the circle method to prove statements about the behavior of $\mathrm{Mor}(\mathbb{P}^{1},X)$ over $\mathbb{C}$ for hypersurfaces $X$ of low degree.  \cite{Bou12} and \cite{Bourqui13} prove related statements in the function field setting using universal torsors.
\end{exam}

\section{Fano threefolds of Picard rank 1 and index 2}

Let $X$ be a smooth Fano threefold such that $\mathrm{Pic}(X) = \mathbb ZH$, $-K_X=2H$ and $H^{3} \geq 2$.  For such varieties the behavior of the $a$ and $b$ constants with respect to subvarieties and covers is understood completely (see \cite{LT16}).  By applying the general theory worked out before, we are able to classify all components of $\Mor(\mathbb{P}^{1},X)$ after making only a few computations in low degree.  The main result in this section, Theorem \ref{theo:componentsforfanothreefolds}, verifies Conjecture \ref{conj:componentsandabreaking} for Fano threefolds of this type.

For the rest of this section, we let $\overline{M}_{0,n}(X, d)$ denote the parameter space of $n$-pointed stable maps whose image has degree $d$ against the ample generator $H$ of $\Pic(X)$.   This space admits an evaluation map
\[
\mathrm{ev}_n : \overline{M}_{0,n}(X, d) \rightarrow X^n.
\]
First we recall the classification of Fano $3$-folds of Picard rank one and index two.

\begin{theo} {\cite[Theorem 3.3.1]{IP99}}
Let $X$ be a smooth Fano $3$-fold with $\mathrm{Pic}(X) = \mathbb ZH$,  $-K_X=2H$, and $H^3 \geq 2$. Then we have $2 \leq H^3 \leq 5$ and the $3$-fold $X$ has the following description:
\begin{itemize}
\item when $H^3 = 5$, $X$ is a section of the Grassmannian $\mathbb{G}(1,4)$ of lines in $\mathbb P^4$ by a general linear subspace of codimension $3$;
\item when $H^3 = 4$, $X$ is a complete intersection of two quadrics in $\mathbb P^5$;
\item when $H^3 = 3$, $X$ is a cubic threefold in $\mathbb P^4$;
\item when $H^3 = 2$, $X$ is a double cover of $\mathbb P^3$ ramified along a smooth quartic surface.
\end{itemize}
\end{theo}

The starting point is to understand the geometric behavior of the $a$ and $b$ invariants:

\begin{lemm} \label{lemm:abclassification}
Let $X$ be a smooth Fano $3$-fold with $\mathrm{Pic}(X) = \mathbb ZH$,  $-K_X=2H$, and $H^3 \geq 2$.

\begin{itemize}
\item There is no subvariety $Y$ with $a(Y,-K_{X}|_{Y}) > a(X,-K_{X})$. 
\item Let $W$ denote the variety of lines on $X$ and let $\mathcal U\rightarrow W$ denote its universal family with the evaluation map $s: \mathcal U \rightarrow X$.  Then $a(\mathcal{U},-s^{*}K_{X}) = a(X,-K_{X})$ and $b(\mathcal{U},-s^{*}K_{X}) = b(X,-K_{X})$.  Furthermore, any dominant thin map $f: Y \to X$ such that $a(Y,-f^{*}K_{X}) = a(X,-K_{X})$ factors rationally through $\mathcal{U}$.
\end{itemize}
\end{lemm}

\begin{proof}
The first statement is verified in \cite[Section 6]{LTT14}.  As for the second statement, it is clear that $s: \mathcal{U} \to X$ satisfies the equality of $a$ and $b$ values and we only need to prove the final claim.  It suffices to consider the case when $Y$ is smooth, and we break into cases based on the Iitaka dimension of the adjoint pair.  If $\kappa(K_{Y} - f^{*}K_{X}) = 2$, then the fibers of the map to the canonical model for this adjoint pair are curves with $a$-value $1$.  Thus their images on $X$ must be lines, and $f$ must factor through $\mathcal{U}$.  If $\kappa(K_{Y} - f^{*}K_{X}) = 1$, then the general fiber $F$ of the canonical map would be a surface with $a$-value $1$ and with $\kappa(K_{F} - f^{*}K_{X}|_{F}) = 0$.  But by the arguments of \cite[Section 6.3]{LTT14} the adjoint pair restricted to such surfaces must have Iitaka dimension $1$, showing that this case is impossible.  Finally, by \cite[Theorem 1.9]{LT16} there is no $a$-cover satisfying $\kappa(K_{Y} - f^{*}K_{X}) = 0$.  
\end{proof}

Let $\alpha$ be a curve class on $X$ such that $H \cdot \alpha = d$.  Based on the computations above, the framework of Section \ref{section:rigidcase} suggests that $\mathrm{Mor}(\mathbb P^1, X, \alpha)$ consists of two irreducible components $R_d, N_d$ such that a general morphism parametrized by $R_d$ is birational and every morphism parametrized by $N_d$ factors through $\mathcal U$. This has been proved for cubic threefolds by Starr (\cite[Theorem 1.2]{CS09}) and for complete intersections of two quadrics by Castravet (\cite{Cas04}). The goal of this section is to verify this expectation for other Fano $3$-folds of Picard rank one and index two. Even though the cases of cubic threefolds and complete intersections of two quadrics are understood, we will provide proofs of these cases as well for completeness.


We need to understand low degree curves on $X$ in order to start the induction.  The next two theorems describe the components of $\overline{M}_{0,0}(X)$ parametrizing curves of $H$-degree $1$ and $2$.
 
\begin{theo}
\label{theo: lines}
Let $X$ be a smooth Fano $3$-fold such that $\mathrm{Pic}(X) = \mathbb ZH$, $-K_X=2H$, and $H^3 \geq 2$.
The space $\overline{M}_{0,0}(X, 1)$ is isomorphic to the variety of lines on $X$.
In particular, it is irreducible and generically parametrizes a free curve.
\end{theo}

\begin{proof}
See \cite[Remark 1.5, Proposition 1.6, and Remark 1.7]{Isk79} for the irreducibility and dominance.
\end{proof}

\begin{prop}
\label{prop: conics}
Let $X$ be a smooth Fano $3$-fold such that $\mathrm{Pic}(X) = \mathbb ZH$, $-K_X=2H$, and $H^3 \geq 2$.
Furthermore when $H^3 = 2$, assume that $X$ is general in its moduli.
Then the space $\overline{M}_{0,0}(X, 2)$ consists of two irreducible components $\mathcal R_2, \mathcal N_2$. Any general element of $\mathcal R_2$ is a stable map from an irreducible curve to a smooth conic and any element of $\mathcal N_2$ is a degree $2$ map from $\mathbb P^1$ to a line.
\end{prop}


\begin{proof}
We study this proposition based on case by case studies.

{\bf Complete intersections of two quadrics in $\mathbb P^5$}:
Let $\mathcal N_2$ be the union of components $M$ of $\overline{M}_{0,0}(X, 2)$ such that for any general element $(C, f)$ of $M$, there is a component of $C$ such that the restriction of $f$ is not birational to its image.
Then $(C, f)$ is a stable map of degree $2$ from $\mathbb P^1$ to a line on $X$.
It is clear that the parameter space $\mathcal N_2$ is irreducible because of Theorem~\ref{theo: lines}.
Let $\mathcal R_2$ be the union of components of $\overline{M}_{0,0}(X,2)$ not contained in $\mathcal N_2$.
By a dimension count, a general element $(C, f)$ on $\mathcal R_2$ is a stable map from $\mathbb P^1$ to a smooth conic contained in $X$. Thus to prove the irreducibility of $\mathcal R_2$, we only need to show that the family of smooth conics is irreducible.

Let $C \subset X$ be a smooth conic. Then there is a unique plane $P$ containing $C$. We denote the pencil of quadrics containing $X$ by $\{Q_\lambda\}_{\lambda \in \mathbb P^1}$.
Then there exists a unique quadric $Q_\lambda$ in this family that contains $P$. 
Indeed, let $q_\lambda$ be a quadric form associated to $Q_\lambda$
and $V$ be a $3$ dimensional vector space associated to $P$.
Since $Q_\lambda$ intersects with $P$ along $C$, the restrictions of $q_0|_V$ and $q_\infty|_V$ are proportional. Thus there exists a unique $q_\lambda$ vanishing identically on $V$.
All smooth conics in $X$ arise in this way, thus we only need to show that the family of planes contained in the quadrics $Q_\lambda$ is irreducible. Each smooth quadric contains two families of planes, and quadrics cones contain one family of planes. Let $\pi: W \rightarrow \mathbb P^1$ be the relative family of planes for $\{Q_\lambda \}$ over $\mathbb P^1$.  Its Stein factorization is a smooth irreducible genus $2$ curve $D \rightarrow \mathbb P^1$. Over $D$ each fiber of $\pi$ is irreducible. Thus $\mathcal R_2$ is irreducible.  


{\bf Cubic threefolds in $\mathbb P^4$}:
We define $\mathcal N_2$ as before and it is easy to see that this is irreducible and parametrizes degree $2$ stable maps from $\mathbb P^1$ to lines.
Let $\mathcal R_2$ be the union of the remaining components.
For a general element $(C, f) \in \mathcal R_2$, $(C, f)$ is a birational map from an irreducible curve to a smooth conic in $\mathbb P^4$. Thus we need to show that the variety of conics is irreducible.
Let $C$ be a smooth conic contained in $X\subset \mathbb P^4$. Then there exists a unique plane $P\subset \mathbb P^4$ containing $C$. The intersection of $X$ and $P$ is the union of a smooth conic and a line. Conversely if we have a line $l \subset X$ and a plane $P$ containing $l$,
then the intersection $X\cap P$ is the union of a conic and a line.
Thus the variety of smooth conics has the structure of a $\mathbb P^2$-bundle over the variety of lines, showing that it is irreducible.


{\bf Double covers of $\mathbb P^3$ ramified along smooth quartics}:
Again we define $\mathcal N_2$ as before and it is easy to see that this is irreducible and parametrizing degree $2$ stable maps from $\mathbb P^1$ to lines.
Let $\mathcal R_2$ be the union of remaining components.
For a general element $(C, f) \in \mathcal R_2$, $(C, f)$ is a birational map from an irreducible curve to a conic in $X$. Thus we need to show that the variety of conics is irreducible.
Let $f: X \rightarrow \mathbb P^3$ be the double cover ramified along a smooth quartic $Y$.
Let $C$ be a conic in $X$.
Then there are two possibilities for $C$:
\begin{itemize}
\item the image of $C$ via $f$ is a line and $C$ is a double cover of the line;
\item the image of $C$ via $f$ is a conic $D$  in $\mathbb P^3$ which is not a double line, and $D$ is tangent to $Y$ at each point of intersection.
\end{itemize}
In the first case, 
since $C$ is rational the line must have at least one point which is tangent to $Y$. However, such lines only form a $3$-dimensional family so the corresponding $C$ cannot form a component of the variety of conics. In the second case, for each conic there exists a unique plane $P$ containing $D$. For each plane $P \subset \mathbb P^3$, consider its intersection $\Gamma_P = P\cap Y$ which is a quartic plane curve and the pullback $f^{-1}(P)$ which is a degree $2$ del Pezzo surface if $\Gamma_P$ is smooth. Conics $C$ corresponding to $P$ are exactly conics in $f^{-1}(P)$ so the variety of conics is a $1$ dimensional family over $(\mathbb P^3)^{*}$: $$\mathcal W \rightarrow (\mathbb P^3)^{*}.$$
Let $\mathcal D \rightarrow (\mathbb P^3)^{*}$ be the Stein factorization which has degree $126$. We would like to show that this $\mathcal D$ is irreducible.

To see this, we use the monodromy action and Lefschetz property developed by Koll\'ar in \cite{Kollar15}. Let $\mathbb P^{14}$ be the space of plane quartic curves in $\mathbb P^2$.
Let $U \subset \mathbb P^{14}$ be the Zariski open set parametrizing smooth curves.
Let $\mathcal D' \rightarrow U$ be the degree $126$ finite cover parametrizing classes of conics on the double cover of $\mathbb{P}^{2}$ ramified along the quartic. 
It is shown in \cite[Example 8.5]{LTT14} that the fundamental group $\pi_1(U)$ acts on a fiber of $\mathcal D' \rightarrow U$ transitively.
Now consider the space $\mathbb P^{34}$ of quartic surfaces in $\mathbb P^3$ and let $V'\subset \mathbb P^{34}$ be the Zariski open set parametrizing irreducible quartic surfaces.
Let $U' \subset \mathbb P^{14}$ be the Zariski open set parametrizing irreducible plane quartics.
We consider the following evaluation map
\[
V' \times \mathrm{PGL}_4 \dashrightarrow U', \quad ([f], [A]) \mapsto [f((x_0, x_1, x_2, 0)A)].
\]
The open subset $C_{V'} \subset V' \times \mathrm{PGL}_{4}$ where the map to $U'$ is defined satisfies all the assumptions of \cite[Theorem 5]{Kollar15}.  For example, the map $C_{V'} \to U'$ is smooth since any fiber is a Zariski open subset of an affine space bundle over $\mathrm{PGL}_4$. Furthermore, since $\mathbb P^{14} \setminus U'$ has codimension $\geq 2$, the general $v \in V'$ will satisfy the Lefschetz property: the map $\pi_{1}((\{v\} \times \mathrm{PGL}_{4})^{0}) \to \pi_{1}(U)$ will be surjective for a suitable open set $(\{v\} \times \mathrm{PGL}_{4})^{0}$. Thus our assertion follows when the quartic surface $Y$ is general.



{\bf Sections of the Grassmannian $\mathbb{G}(1,4)$}: 
Again we define $\mathcal N_2$ as before and it is easy to see that this is irreducible and parametrizing degree $2$ stable maps from $\mathbb P^1$ to lines.
Let $\mathcal R_2$ be the union of remaining components.
For a general element $(C, f) \in \mathcal R_2$, $(C, f)$ is a birational map from an irreducible curve to a conic in $X$. Thus we need to show that the variety of conics is irreducible.
The result follows from \cite[Proposition 2.32]{San13}. 
\end{proof}


We are now ready to describe the induction.  Our approach is motivated by \cite{CS09}.  We start with an auxiliary lemma:

\begin{lemm} \label{lemm:pointdef}
Let $Y$ be a projective variety of dimension $n$ and let $\phi: Y' \to Y$ be a resolution.  Fix a point $p \in Y$ and suppose that there is a dominant family of rational curves through $p$ on $Y$ parametrized by a variety $W \subset \overline{M}_{0,0}(Y)$.  Let $C'$ denote the strict transform of a general curve in the family to $Y'$.  If the fiber of $\phi$ over $p$ has dimension $k$, then $\dim(W) \leq -K_{Y'} \cdot C' - 2 + k$.
\end{lemm}

\begin{proof}
Let $W'$ be the variety parametrizing deformations of the strict transform $C'$.  Since this family dominates $Y'$, a general member is free.  Thus the dimension of the sublocus $W'$ parametrizing curves through a general fixed point $p'$ in the preimage of $p$ has at most the expected dimension $-K_{Y'} \cdot C' - 2$.  The statement is now clear.
\end{proof}

\begin{theo} \label{theo:threefoldfinitebadpoints}
Let $X$ be a smooth Fano $3$-fold such that $\mathrm{Pic}(X) = \mathbb ZH$, $-K_X=2H$, and $H^3 \geq 2$.
Furthermore when $H^3 = 2$, assume that $X$ is general in its moduli.
Let $\alpha$ denote a nef curve class and let $d$ denote its anticanonical degree.  Suppose that $W$ is a component of $\overline{M}_{0,0}(X,\alpha)$ and let $W_{p}$ denote the sublocus parametrizing curves through the point $p \in X$.  There is a finite union of points $S \subset X$ such that
\begin{itemize}
\item $W_{p}$ has the expected dimension $d - 2$ for points $p$ not in $S$, and
\item $W_{p}$ has dimension at most $d - 1$ for points $p \in S$.
\end{itemize}
Furthermore for $p \not \in S$ the general curve parametrized by $W_{p}$ is irreducible.
\end{theo}

\begin{proof}
By \cite[II.3.5.4 Corollary and II.3.10.1 Corollary]{Kollar}, there is a proper closed subset $Q \subsetneq X$ which contains any non-free component of a member of the family of curves in $W$.

The proof is by induction on $d$.  The base case is when $W$ is the family of lines.  By Theorem \ref{theo: lines} $W$ is irreducible and has dimension $2$, so there are only finitely many points in $X$ which are contained in a one-parameter family of lines.  We let $S$ denote this finite set.

We now prove the induction step.  Let $W'$ be a component of $W_{p}$.  First suppose that the general curve parametrized by $W'$ is irreducible.  If the general curve parametrized by $W'$ is free, then $W'$ has the expected dimension, so we may suppose otherwise.  We divide into cases based on the dimension of the subvariety swept out by curves in $W'$.

Suppose that the curves parametrized by $W'$ all map to a fixed curve $Y \subset X$. 
If the general curve maps $r$-to-$1$, then the dimension of $W'$ is at most $2r-2$.  Note that $d = r (-K_{X} \cdot Y)$.  Comparing against the expected dimension $-K_{X} \cdot C - 2$, we see that $W'$ can have larger than expected dimension only when $-K_{X} \cdot Y = 1$.  However, in this case we have $a(Y,-K_{X}|_{Y}) = 2$, violating Lemma \ref{lemm:abclassification}. 

Next suppose that the curves parametrized by $W'$  dominate an irreducible projective surface $Y \subset Q$.
Let $\phi: Y' \to Y$ be a resolution; applying Lemma \ref{lemm:pointdef} we see that
\begin{equation*}
\dim(W') \leq -K_{Y'} \cdot C' - 1.
\end{equation*}
If $\dim(W') \geq -K_{X} \cdot C$, then by the equation above we see that $(K_{Y'} - \phi^{*}K_{X}) \cdot C'$ is negative.  This implies that $a(Y,-K_{X}) > 1$, an impossibility by Lemma \ref{lemm:abclassification}.  This proves that $\dim(W_{p}) \leq -K_{X} \cdot C - 1$.  We also need to characterize the equality case.  When $K_{Y'} \cdot C = K_{X} \cdot C$, it follows that $a(Y',-K_{X}) = 1$.  By the earlier classification this means that $K_{Y'} - K_{X}$ has Iitaka dimension $1$.  Since $C$ has vanishing intersection against this divisor, it is a fiber of the Iitaka fibration.  We conclude that $2 = -K_{Y'} \cdot C = -K_{X} \cdot C$. By Theorem~\ref{theo: lines}, this means that $p \in S$. 

Now suppose that the curves parametrized by $W'$ are reducible.  The dimension counting arguments of \cite[Proposition 2.5]{CS09} show how to deduce the desired conclusion for curves parametrized by $W'$ from the same properties for their irreducible components.  For clarity we will outline the argument here.
Assume that our assertions hold for any stable maps of degree less than $d$.
Suppose that $p \not\in S$ and let $f: C \rightarrow X$ be a general member of $W'$.
We analyze case-by-case:
\begin{enumerate}
\item Suppose that a node of $C$ maps to the point $p$. Let $D$ be a maximal connected subset of $C$ contracted to the point $p$. Let $C_1, \cdots, C_u$ be the closures of the connected components of $C \setminus D$.  Set $d_i =K_X.C_i$. Then the induction hypothesis implies that the dimension of $W'$ is bounded by
\[
\sum_{i = 1}^u (d_i-2) + u-2
\]
where the term $u-2$ accounts for the dimension of the marked point and the points of attachment of $D$ with $C_i$. Since $d = \sum_i d_i$, we conclude that the dimension of $W'$ is bounded by
\[
d-u-2.
\]
In particular $W'$ can not have larger than the expected dimension.  
\item Suppose that a node of $C$ maps to a point $p_i$ contained in $S$. Let $D$ be the maximal connected subset of $C$ contracted to the point $p_i$. Let $C_1, \cdots, C_u$ be the closures of the connected components of $C \setminus D$ and let $d_i$ be the degree of $C_i$. Suppose that the inverse image of $p$ is contained in $C_1$. Then by the induction hypothesis, the dimension of $W'$ is bounded by
\[
d_1-2 + \sum_{i = 2}^u d_i-1 + \max\{0, u-3\} = d-u -1 +\max\{0, u-3\}
\]
where the term $\max\{0, u-3\}$ accounts for the dimension of the moduli of the points of attachment.  
In particular $W'$ can not have larger than the expected dimension.  
\item Suppose that a node of $C$ maps to a point $q\not\in \{p\} \cup S$. Let $D$ be the maximal connected subset of $C$ contracted to the point $q$. Let $C_1, \cdots, C_u$ be the closures of the connected components of $C \setminus D$ and let $d_i$ be the degree of $C_i$. Suppose that the inverse image of $p$ is contained in $C_1$. Then by the induction hypothesis, the dimension of $W'$ is bounded by
\[
d_1-2 + \sum_{i = 2}^u (d_i-2) + \max\{0, u-3\} + 1 = d -2u + 1 + \max\{0, u-3\}
\]
where the term $\max\{0, u-3\} + 1$ accounts for the dimension of the moduli of the points of attachment. In particular $W'$ can not have larger than the expected dimension. 
\end{enumerate}

Together these three cases exhaust all possibilities for the node of $C$, and our claim follows when $p \not\in S$. The case of $p \in S$ is similar; we refer readers to \cite[Proposition 2.5]{CS09} for more details.  In particular, the above discussion shows that the general curve parametrized by $W'$ for points $p \not \in S$ is irreducible.
\end{proof}

We see as an immediate consequence:

\begin{coro}
\label{coro:free}
Let $X$ be a smooth Fano $3$-fold such that $\mathrm{Pic}(X) = \mathbb ZH$, $-K_X=2H$, and $H^3 \geq 2$.
Furthermore when $H^3 = 2$, assume that $X$ is general in its moduli.  For any $\mathbb{Z}$-curve class $\alpha$, if $\overline{M}_{0,0}(X,\alpha)$ is non-empty then every component generically parametrizes free curves and has the expected dimension.
\end{coro}

Applying the arguments of \cite[Corollary 2.6 and Corollary 2.7]{CS09}, we now obtain:

\begin{theo} \label{theo:polygrowthofmbar}
Let $X$ be a smooth Fano $3$-fold such that $\mathrm{Pic}(X) = \mathbb ZH$, $-K_X=2H$, and $H^3 \geq 2$.
Furthermore when $H^3 = 2$, assume that $X$ is general in its moduli.  Then any free curve on $X$ deforms into a chain of free curves of anticanonical degree $\leq \dim(X)+1$.
\end{theo}

\begin{proof}
By Mori's Bend and Break any free curve $C$ of anticanonical degree $> \dim(X) + 1$ can be deformed to a stable map with reducible domain.  Furthermore these deformed maps form a codimension $1$ locus of the component of the moduli space containing $C$.  By the classification of components of $\overline{M}_{0,0}(X)$, only the union of two free curves can form such a codimension $1$ locus. Now Lemma~\ref{lemm: chains} and the induction argument show our claim.
\end{proof}

\begin{theo}
\label{theo:componentsforfanothreefolds}
Let $X$ be a smooth Fano $3$-fold such that $\mathrm{Pic}(X) = \mathbb ZH$, $-K_X=2H$, and $H^3 \geq 2$.
Furthermore when $H^3 = 2$, assume that $X$ is general in its moduli.
Let $d \geq 2$. Then $\overline{M}_{0,0}(X, d)$ consists of two irreducible components:
\[
\overline{M}_{0,0}(X, d) = \mathcal R_d \cup \mathcal N_d
\]
such that a general element $(C, f) \in \mathcal R_d$ is a birational stable map from an irreducible curve and any element $(C, f) \in \mathcal N_d$ is a degree $d$ stable map to a line in $X$. 
Moreover the fiber $\mathrm{ev}_1^{-1}(x)\cap \mathcal R_d'$ is irreducible for a general $x\in X$.
\end{theo}

\begin{proof}
We denote by $\mathcal N_{d}$ the component parametrizing degree $d$ stable maps to lines.
It is clear that this is irreducible.
First let $M$ be a dominant component of $\overline{M}_{0,0}(X)$ generically parametrizing a birational stable map. We show that the fiber $\mathrm{ev}_1^{-1}(x)\cap M$ is irreducible for a general $x\in X$. Let $Y \rightarrow M$ be a smooth resolution.
We would like to show that $Y \rightarrow X$ has connected fibers.
Suppose not, i.e. the Stein factorization $Z \rightarrow X$ is nontrivial.
Then it is shown in \cite{LT16} that $Z$ factors through $\mathcal R_1'$ rationally.
This means that curves parametrized by $M$ lift to $\mathcal R_1'$ and have vanishing intersection with the ramification divisor of the morphism $\mathcal R_1' \rightarrow X$.  Therefore, these curves are multiple covers of lines which contradicts with the fact that a general curve maps birationally onto its image. Thus our assertion follows.

Now we prove our theorem using induction on $d$.
The case of $d=2$ is settled by Proposition~\ref{prop: conics}. Suppose that $d > 2$ and we assume our assertion for any $2 \leq d' < d$.  By gluing free curves of lower degree it is clear that $\overline{M}_{0,0}(X,d)$ has at least one component different from $\mathcal{N}_{d}$.  Let $M$ be one such component.
Then any general $(C, f) \in M$ is a birational stable map from an irreducible curve by Corollary \ref{coro:free}.  (A dimension count shows that multiple covers of curves can not form a component of $\overline{M}_{0,0}(X)$ unless the curves are lines.) 
Using Theorem~\ref{theo:polygrowthofmbar} we see that $M$ contains a chain of free curves of degree at most $2$.  Furthermore, by Proposition~\ref{prop: conics} each component of the parameter space of conics contains a chain of lines.  Applying Lemma~\ref{lemm: chains}, we can conclude that $M$ contains a chain $(C, f)$ of free lines of length $d$.
Note that $(C, f)$ is a smooth point of $M_{0,0}(X)$.
If the image of $f$ is a single line, then $(C, f)$ is contained in $\mathcal{N}_d$.
Since it is a smooth point, we conclude that $M = \mathcal{N}_d$ which is a contradiction.

So we may assume that $(C, f)$ has a reducible image.
This means that $(C, f)$ is a point on the image $\Delta_{1, d-1}$ of the main component of $\mathcal R_1' \times_X \mathcal R_{d-1}'$ which is unique by the induction hypothesis.
Since $(C, f)$ is a smooth point, we conclude that $M$ contains $\Delta_{1, d-1}$
and such $M$ must be unique.
Thus our assertion follows.

\end{proof}

\begin{coro}
Let $X$ be a smooth Fano $3$-fold such that $\mathrm{Pic}(X) = \mathbb ZH$, $-K_X=2H$, and $H^3 \geq 2$.
Furthermore when $H^3 = 2$, assume that $X$ is general in its moduli.  For every curve class $\alpha$ satisfying $H \cdot \alpha \geq 2$ there is a unique Manin component representing $\alpha$.
\end{coro}



\bibliographystyle{alpha}
\bibliography{balancedIV}

\end{document}